\documentclass[a4paper,11pt,reqno]{amsart}

\usepackage[T1]{fontenc}			
\usepackage[utf8]{inputenc}
\usepackage{amsmath, amsfonts, amssymb, amsthm, amscd}
\usepackage[foot]{amsaddr} 
\usepackage[a4paper,scale={0.72,0.75},marginratio={1:1},footskip=7mm,headsep=10mm]{geometry}
\usepackage[english]{babel}
\usepackage{csquotes}
\usepackage[square,numbers]{natbib}	
\usepackage{bbm}
\usepackage{bm}
\usepackage{mathtools}
\usepackage{xspace} 
\usepackage{booktabs}
\usepackage{enumerate}
\usepackage{scalerel}
\usepackage{stackengine,wasysym}
\usepackage
[pdftitle={Filtering of pure jump processes with partial observation},
pdfauthor={Elena Bandini, Alessandro Calvia, Katia Colaneri},%
hidelinks,%
pdfstartview={XYZ null null 1.4}]{hyperref}
\usepackage{verbatim} 
\usepackage{mathrsfs}
\usepackage[monochrome]{xcolor}

\usepackage{stmaryrd}

\mathtoolsset{showonlyrefs}

\definecolor{gre}{rgb}{0.03,0.50,0.03}

\newcommand{\bbE}{{\ensuremath{\mathbb E}} }
\newcommand{\bbF}{{\ensuremath{\mathbb F}} }
\newcommand{\bbG}{{\ensuremath{\mathbb G}} }
\newcommand{\bbH}{{\ensuremath{\mathbb H}} }

\newcommand{\bbP}{{\ensuremath{\mathbb P}} }

\newcommand{\bbY}{{\ensuremath{\mathbb Y}} }

\newcommand{\cA}{{\ensuremath{\mathcal A}} }
\newcommand{\cB}{{\ensuremath{\mathcal B}} }

\newcommand{\cD}{{\ensuremath{\mathcal D}} }
\newcommand{\cE}{{\ensuremath{\mathcal E}} }
\newcommand{\cF}{{\ensuremath{\mathcal F}} }
\newcommand{\cG}{{\ensuremath{\mathcal G}} }
\newcommand{\cH}{{\ensuremath{\mathcal H}} }

\newcommand{\cL}{{\ensuremath{\mathcal L}} }

\newcommand{\cN}{{\ensuremath{\mathcal N}} }

\newcommand{\cP}{{\ensuremath{\mathcal P}} }

\newcommand{\cY}{{\ensuremath{\mathcal Y}} }
\newcommand{\cZ}{{\ensuremath{\mathcal Z}} }

\newcommand{\bfb}{{\ensuremath{\mathbf b}} }

\newcommand{\bfm}{{\ensuremath{\mathbf m}} }

\newcommand{\bfG}{{\ensuremath{\mathbf G}} }

\newcommand{\bfK}{{\ensuremath{\mathbf K}} }

\newcommand{\bfQ}{{\ensuremath{\mathbf Q}} }
\newcommand{\bfR}{{\ensuremath{\mathbf R}} }

\newcommand{\bfX}{{\ensuremath{\mathbf X}} }

\newcommand{\bflambda}{{\ensuremath{\boldsymbol \lambda}}}
\newcommand{\bfmu}{{\ensuremath{\boldsymbol \mu}}}

\newcommand{\dd}{{\ensuremath{\mathrm d}} }

\newcommand{\dA}{{\ensuremath{\mathrm A}} }
\newcommand{\dB}{{\ensuremath{\mathrm B}} }

\newcommand{\dI}{{\ensuremath{\mathrm I}} }

\newcommand{\dL}{{\ensuremath{\mathrm L}} }

\newcommand{\R}{\mathbb{R}}

\newcommand{\N}{\mathbb{N}}

\renewcommand{\P}{\mathbb{P}}

\newcommand{\ind}{\ensuremath{\mathbf{1}}}
\newcommand{\one}{\ensuremath{\mathsf{1}}}

\DeclarePairedDelimiterX{\abs}[1]{\lvert}{\rvert}{#1}
\DeclarePairedDelimiterX{\norm}[1]{\lVert}{\rVert}{#1}
\renewcommand{\epsilon}{\varepsilon}
\let\temp\theta
\let\theta\vartheta
\let\vartheta\temp

\newcommand{\phivarphi}
{
\let\temp\phi
\let\phi\varphi
\let\varphi\temp
}

\theoremstyle{plain}
\newtheorem{theorem}{Theorem}[section]
\newtheorem*{theorem*}{Theorem}
\newtheorem{lemma}[theorem]{Lemma}
\newtheorem*{lemma*}{Lemma}
\newtheorem{proposition}[theorem]{Proposition}
\newtheorem*{proposition*}{Proposition}
\newtheorem{corollary}[theorem]{Corollary}
\newtheorem{assumption}[theorem]{Assumption}

\theoremstyle{definition}

\newtheorem{example}[theorem]{Example}

\theoremstyle{remark}
\newtheorem{remark}[theorem]{Remark}
\newtheorem*{remark*}{Remark}

{\left\lbrace\begin{array}{@{}l@{}}}%
{\end{array}\right.}

\numberwithin{equation}{section}

\newcommand{\mail}[1]{\href{mailto:#1}{\normalfont\texttt{#1}}}

\makeatletter
\def\@setthanks{\vspace{-\baselineskip}\def\thanks##1{\@par##1\@addpunct.}\thankses}
\makeatother

\definecolor{verde}{rgb}{0,0.7,0}

\title[Stochastic filtering with path-dependent local characteristics]
	{Stochastic filtering of a pure jump process with predictable jumps and path-dependent local characteristics}
\author[E.~Bandini]{Elena Bandini\textsuperscript{\MakeLowercase{a},1}}
\thanks{\noindent \textsuperscript{a} Universit\`a di Bologna, Dipartimento di Matematica, 
Bologna (Italy).}
\author[A.~Calvia ]{Alessandro Calvia\textsuperscript{\MakeLowercase{b},2}}
\thanks{\noindent \textsuperscript{b} LUISS University, Department of Economics and Finance, Rome (Italy).}
\author[K.~Colaneri ]{Katia Colaneri\textsuperscript{\MakeLowercase{c},3}}
\thanks{\noindent \textsuperscript{c} University of Rome Tor Vergata, Department of Economics and Finance, Rome (Italy).
\\
\noindent \textsuperscript{1} E-mail: \mail{elena.bandini7@unibo.it}.
\\
\noindent \textsuperscript{2} E-mail: \mail{acalvia@luiss.it}.
\\
\noindent \textsuperscript{3} E-mail: \mail{katia.colaneri@uniroma2.it}
\\
\noindent This research was partially supported by the 2018 GNAMPA-INdAM project \textit{Controllo ottimo stocastico con osservazione parziale: metodo di randomizzazione ed equazioni di Hamilton-Jacobi-Bellman sullo spazio di Wasserstein}.}

\begin{document}

\phivarphi
	
\begin{abstract}
The objective of this paper is to study the filtering problem for a system
of partially observable processes (X, Y), where X is a non-Markovian pure jump process
representing the signal and Y is a general jump diffusion which provides observations.
Our model covers the case where both processes are not necessarily quasi left-continuous,
allowing them to jump at predictable stopping times. By introducing the Markovian version
of the signal, we are able to compute an explicit
equation for the filter via the innovations approach.
\end{abstract}
\maketitle
	

\smallskip

\noindent \textbf{AMS 2020:} 60G35, 60G57, 60J60, 60J76.

\smallskip

\section{Introduction} \label{sec:intro}
In this paper we study a stochastic filtering problem for a partially observable system where the unobservable signal process is a pure-jump process, possibly non-Markovian, and takes values in a complete and separable metric space. The observation process is a jump-diffusion with local characteristics that depend on the signal. Our goal is to derive the conditional distribution of the signal process given the available information flow generated by the observation process, i.e., to compute the \emph{filter}. In other terms we aim to  write an explicit  evolution equation satisfied by the  filter, called \emph{filtering equation}. Our contribution is methodological in nature: while some examples are shown and some possible applications are briefly outlined, our main intent is to provide a general partially observed stochastic model and to rigorously compute the corresponding filtering equation. The partially observed system considered in this paper has the following two \textcolor{blue}{novel} features:
\begin{itemize}
\item A non-Markovian pair signal-observed process, which is realized by allowing the data of our model (i.e., local characteristics) to be path-dependent with respect to the signal process.
\item Predictable jump times for both the signal and the 
observation, to account for models where such jump times arise naturally (for instance, Piecewise Deterministic Processes or reflecting diffusions).
\end{itemize}

\subsection{Motivation and literature review}
Stochastic filtering in continuous-time is a classic branch of research in applied probability and has been studied under several choices of the pair signal-observation. Starting from the pioneering works of N.~Wiener and R.~E.~Kalman (with its celebrated \emph{Kalman filter}), research on this subject has vastly expanded in numerous directions and is still ongoing. A fairly detailed account of these developments can be found, for instance in \citet[Chapter 1]{bain:fundofstochfilt}. In this section we summarize the contributions that are more related to our paper.

One of the driving motivations to study stochastic filtering is that \emph{partially observed systems} naturally arise in fields like engineering, operations research, economics, and finance. These stochastic models are characterized by a lack of information (so called partial information or partial observation), as some of the quantities involved in the model itself are not directly measurable and need to be estimated from the observable quantities.

From a technical perspective, various situations have been considered so far in the literature: linear and nonlinear, finite and infinite dimensional filters (even on manifolds), numerical schemes, and so on.
One of the most ubiquitous assumptions in this literature is that the pair signal-observation solves a martingale problem for some infinitesimal generator, implying that it is a Markov process.

This fundamental assumption allows to obtain an {\em explicit} filtering equation. 
Such characterization of the filter has important implications. First, it is key to estimate statistics of the unobserved process. Second, although in general filtering equations do not admit solutions in closed form, one can  resort to numerical schemes to approximate such solutions,  which require at least explicit filtering equations, see e.g. \citet{bain:fundofstochfilt}, \citet{damian2018algorithm}.
 Third, to solve optimal control problems with partial observation, one usually needs to have an explicit filtering equation and to apply the so-called \emph{separation principle}, which enables to switch from the original optimization problem under partial information to an equivalent optimal control problem with complete information, where the hidden state  is replaced by the filter, see, e.g., \citet{calvia2018optimal,altay:pairstrading, colaneri:optliquid, calvia2022Nonlinear}.
 Under the Markovianity hypothesis, several filtering problems have been addressed. 
Partially observable systems with jump-diffusion observation are studied, for instance, in \citet{grigelionis2011, cecicolaneri:zakai} and the cases of  pure-jump signal are discussed, for instance, by \citet{calvia:filtcontrol}, \citet{ceci2000filtering,ceci2001nonlinear}.

In contrast to the classical setting, in a non-Markovian context results available in the literature appear to be more abstract in nature. Although general equations for this case are available, see, e.g., \citet{szpirglas1979modele,kallianpur1980:stochfilt,liptsershiryaev2001:statistics},
 in these works the partially observable system is modeled via an abstract semimartingale decomposition, and consequently,
the underlying partially observed models,
and hence the filtering equation, are not explicit.

General non-Markovian partially observed models have potential valuable applications. In fact, there are interesting and well-known problems where state variables depend on (part of) the history of some non-directly observable stochastic process, and not exclusively on its current value. They arise, for instance, in models with delay or when dealing with exotic options in finance and in various optimal control problems with partial observation (see, e.g., \citet{bandini2018:randBSDE,bandini:randomizHJBwasserstein,tang1998:smp}).


In light of this, our first contribution is to set a partially observed model where the pair signal-observation can be non-Markovian and to compute an explicit filtering equation in this context. The lack of the Markov property is due to dependence of our modeling data (i.e., the local characteristics) on the path of the signal process. The signal is a pure-jump process which is allowed to take values in a fairly general state space (metric, complete, and separable). This is in contrast with the usual setting where the state space of the signal is either discrete and finite or Euclidean.


In our model the pair signal-observation has common jumps times, which may also be \emph{predictable}. To best of our knowledge, this feature has not been yet considered in the literature, where the common assumption is to take only \emph{totally inaccessible}  jump times (cf. the definitions in \citet[Chapter I]{jacod2013:limit}, see also Section \ref{sec:prelim}).
Nevertheless, predictable jumps naturally appear in various stochastic models and even in the definition of some classes of processes, as in the well-known case of Piecewise Deterministic Processes, introduced by M.~Davis \cite{davis:markovmodels}. In the financial literature, for instance, recent works discussed modelization of real financial markets presenting critical announced random dates, typically related to possible situations of default (see, e.g., \citet{fontana2018:termstruct, jiao2015:gendens} and discussions therein) or to information delivered at a priori established dates (e.g., central banks updates of interest rates).



Let us now briefly outline some of the details of our analysis.

\subsection{Main results and structure of the paper}
We consider a pair of stochastic processes $(X,Y) = (X_t, Y_t)_{t \geq 0}$, respectively called the \emph{signal process} and the \emph{observed process}, on some  probability space \color{blue}endowed with a global filtration $\bbF=(\cF_t)_{t \geq 0}$. Processes $X$ and $Y$ are $\bbF$-adapted; however, the available information is only given by the subfiltration $\bbY = (\cY_t)_{t \geq 0}$ which is  the completed and right-continuous natural filtration of $Y$. This means that the signal is not observable, and the information on the signal can only be retrieved through the process $Y$ whose dynamics depends, more or less directly, on $X$.
\normalcolor
We assume that $X$ is a pure-jump process, possibly non Markovian, described in terms of a random counting measure (see, e.g., \citet{bremaud:pp, jacod:mpp, jacod2013:limit}). In contrast to most of the existing literature, we do not assume that the dual predictable projection, or \emph{compensator}, of $X$ is \mbox{quasi-left} continuous and we allow for predictable jumps.
Our objective is to characterize the filter $\pi = (\pi_t)_{t \geq 0}$, 
\color{blue}which permits to describe the conditional distribution of the path of $X$ up to time $t$, given the available information $\cY_t$. 
Therefore, to compute it we need to consider functionals depending on the history of the signal process, which is another novelty of our paper.

Due to the lack of Markovianity of the signal, to solve the filtering problem 
we construct  the \emph{history process} $\bfX$.
\normalcolor This is an auxiliary process which keeps track of all past values and jump times of $X$ and has the advantage of being Markovian with respect to its natural filtration. 
\color{blue} Moreover, such process shares the same pure jump nature as the original signal, and it is fully determined by its local characteristics, which can be computed in terms of those of the signal (see Proposition \ref{prop:bfmudualpredproj}).
\normalcolor
The existence of a bijective function that maps $X$ into $\bfX$ and vice versa (see Proposition~\ref{P_S}), permits to characterize the filter $\pi$ by addressing an equivalent filtering problem 
\color{blue}where the conditional distribution of the history process $\bfX$ is derived (see Lemma \ref{L:4.1}).

In this paper, we consider a quite general observation process $Y$  which follows a jump-diffusion, whose local characteristics depend on the trajectory of the signal $X$, and may also have 
predictable jumps. We allow for the signal and the observation to have common (both predictable and totally inaccessible) jump times. Having such structure for the observation process permits us to model several 
 information
flows. For instance, if $X$ and $Y$ have only common jump times, then the filter would be very informative and able to detect all jump times of the signal, although its positions would still be unknown. Considering the case where $Y$ has also disjoint jump times and a diffusion component, brings additional noise, and the filter is not necessarily able to identify all the 
jumps of the signal.
\normalcolor

To solve the filtering problem we resort to the innovations approach, see, e.g. \cite{bain:fundofstochfilt, bremaud:pp, kallianpur1980:stochfilt}. This is a classical technique in filtering theory which is particularly convenient when signal and observation have common jump times.
Although the idea of the innovations approach is well known, it is anything but easy to apply it in our setting, 
\color{blue}due to the generality of our partially observable system.
\normalcolor

The paper is organized as follows. This introduction concludes with a brief paragraph on the notations and conventions adopted here.
\color{blue}In Section \ref{sec:model} we describe the partially observable system that we intend to analyze and we give all the assumptions. 
\normalcolor 
In Section \ref{sec:markov} we introduce the history process, explain why it is useful to do so, and  write the model previously described in this new setting. \color{blue} The filter is introduced in  Section \ref{sec:filtprob} and \normalcolor
 in Section \ref{sec:martrepr} we provide the martingale representation theorem with respect to the filtration generated by the observed process, that is fundamental to derive the filtering equation. The latter is computed in Section \ref{sec:filter}. \color{blue} Finally, three illustrative examples are collected in Section \ref{sec:examples}. \normalcolor 
  The proofs of some results stated in the paper are gathered in Appendix.  

\subsection{Notation}\label{sec:notation}
In this section we collect the main notation used in the  paper.

Throughout the paper the set $\N$ denotes the set of natural integers $\N = \{1, 2, \dots \}$ and  $\N_0= \N \cup \{0\}$. 

We indicate by $\cN$ the collection of null sets in some specified probability space.

The symbol $\ind_C$ denotes the indicator function of a set $C$, while $\one$ is the constant function equal to $1$. The symbol $\int_a^b$ denotes $\int_{(a,b]}$ for any $-\infty < a \leq b < +\infty$.

For a fixed metric space $E$, we denote by $d_E$ its metric and by $\dB_b(E)$ the set of real-valued bounded measurable functions on $E$. The symbol $\cB(E)$ indicates the Borel $\sigma$-algebra on $E$ and we denote by $\cP(E)$ the set of probability measures on $E$. The set of $E$-valued c\`adl\`ag functions on $[0,+\infty)$ is denoted by $\cD_E$. We always endow it with the Skorokhod topology and the Borel $\sigma$-algebra. The set $\widetilde \cD_E \subset \cD_E$ contains all trajectories in $\cD_E$ that are piecewise constant and do not exhibit explosion in finite time (i.e., if $(t_n)_{n \in \N} \subset (0,+\infty]$ is the collection of discontinuity points of some trajectory, then $\lim_{n \to \infty} t_n = +\infty$).

{\color{blue}
For any given $E$-valued c\`adl\`ag stochastic process $\eta = (\eta_t)_{t \geq 0}$ defined on a probability space $(\Omega, \cF, \bbP)$, we denote by $(\eta_{t^-})_{t \geq 0}$ the left-continuous version of $\eta$ (i.e., $\eta_{t^-} = \lim_{s \to t^-} \eta_s, \, \bbP$-a.s., for any $t \geq 0$). The notation $\eta_{t \land \cdot}$ indicates the path of process $\eta$ stopped at time $t \geq 0$, i.e., $(\eta_{t \land \cdot})_{t \geq 0}$ is the $\cD_E$-valued stochastic process such that, for any $t \geq 0$, $\eta_{t \land \cdot} = \{s \mapsto \eta_{t \land s}\}_{s \geq 0}$. If $\eta$ is real-valued, that is $E = \R$, we denote by $\Delta \eta_t \coloneqq \eta_t - \eta_{t^-}$ the jump size of $\eta$ at time $t \geq 0$.
}

Finally, with the word measurable we refer to Borel-measurable, unless otherwise specified.

\section{The model} \label{sec:model}
{\color{blue}The aim of this section is to introduce the model for the partially observed system that we aim to study, which is composed of a pair of processes, respectively called the signal and the observed process. Recall that our objective is to derive the conditional distribution of the signal process given the information provided by the observed process. Put in other words, we aim to characterize the dynamics of the filter, which is rigorously defined in~\eqref{eq:pidef}, equivalently to provide the filtering equation.}

\subsection{Preliminaries}\label{sec:prelim}
Throughout the paper we are given a complete and separable metric space $E$, that serves as the state space of the signal process.
We endow it with its Borel $\sigma$-algebra $\cB(E)$ and we denote its metric by $d_E$. We indicate by $\cD_E$ the set of $E$-valued c\`adl\`ag functions on $[0,+\infty)$.
{\color{blue}We also consider an auxiliary Lusin space $(Z, \cZ)$.}

We are also given a   probability space $(\Omega, \mathcal F,  \bbP)$ on which  is defined  a complete right-continuous filtration $\bbF \coloneqq (\cF_t)_{t \geq 0}$.  This filtered probability space carries
a pair of \textcolor{blue}{continuous-time} stochastic processes $\color{blue}(X,Y)= (X_t, Y_t)_{t \geq 0}$, where $X$ {\color{blue}is the signal process, which is not directly observable, and $Y$ is the observed process, which provides noisy information about the signal.

We denote by $\bbY=(\cY_t)_{t \geq 0}$ the complete natural filtration of the process $Y$, that is $\cY_t=\sigma(Y_s, \ 0\leq s \leq t)\vee \cN$, where $\cN$ indicates the collection of $\P$-null sets. 
We assume that $\bbY$ is right-continuous, otherwise we consider the right-continuous enlargement of $(\cY_t)_{t \geq 0}$,  still denoted by $\bbY$ (see, e.g., \citep{bain:fundofstochfilt, KurtzOcone}); it contains all  available information that we can use to infer the distribution signal process $X$.

We  recall the following definitions. 
A stopping time $\tau$ on $(\Omega, \mathcal F, \bbF, \bbP)$ is said to be predictable if and only if there is a sequence of stopping times $(\tau_n)_{n \in \mathbb{N}}$ increasing to $\tau$, called an announcing sequence of $\tau$, such that $\tau_n<\tau$, $\mathbb{P}$-a.s., on $\{\tau>0\}$.
A stopping time $\tau$  is called totally inaccessible if $\mathbb{P}(\tau = \eta < + \infty)=0$
 for any predictable time $\eta$. 
For further details on these concepts see for instance \citep[Chapter~I, Sections~2b and~2c, pages 19-20]{jacod2013:limit}.  
\normalcolor

\subsection{The partially observed system}\label{sec:modelassumption}
{\color{blue}In this section we introduce assumptions on the partially observed system $(X,Y)$, under which
\begin{itemize}
\item 
the process $X$ is an $E$-valued $\bbF$-adapted \emph{pure-jump} process with c\`{a}dl\`{a}g paths;
\item the process $Y$ is a \mbox{real-valued} $\bbF$-adapted \mbox{\emph{jump-diffusion}} process with c\`{a}dl\`{a}g paths and common jumps times with the process $X$.
\end{itemize}

Let $(T_k)_{k \geq 1}$ be a sequence of $\mathbb F$-stopping times and  $(\zeta_k)_{k \geq 0}$ be a sequence of $E$-valued random variables, defined on $(\Omega, \cF, \bbP)$, such that (see, e.g., \citet{boel_varaiya_wong}):
\begin{enumerate}[(i)]
\item $T_k \leq T_{k+1}$ and, on $\{T_k < +\infty\}$, $T_k < T_{k+1}$, for all $k \geq 1$;
\item $\zeta_k$ is $\cF_{T_k}$-measurable, for each $k \geq 0$.
\end{enumerate}
The we define the process $X$ as follows
\begin{equation}\label{eq:Xrepr}
X_t=\zeta_0 \ind_{[0,T_1)}(t) + \sum_{k=1}^\infty \zeta_n \ind_{[T_k,T_{k+1})}(t), \quad t \geq 0.
\end{equation}
We can always associate to $X$ the random counting measure $m$ on $(0,+\infty) \times E$ defined as
\begin{equation}\label{eq:m}
	m\big ((0, t] \times A\big ) = \sum_{k =1}^\infty 1_{T_k \leq t} \, 1_{X_{T_k} \in A}, \quad t \geq 0, \, A \in \mathcal B(E),
\end{equation}
with $\bbF$-dual predictable projection $\mu$ (see, e.g., \cite[Chapter II.1]{jacod2013:limit}, \cite[Chapter VIII]{bremaud:pp}).

We suppose, in addition, that the filtered probability space $(\Omega, \cF, \bbF, \bbP)$ supports a real-valued, $\bbF$-adapted standard Brownian motion $W$ and an $\bbF$-adapted random counting measure $n$ on $(0,+\infty) \times Z$, where $Z$ is the auxiliary Lusin space (see Section~\ref{sec:prelim}), with $\bbF$-dual predictable projection $\nu$.

Next, we introduce the following measurable functions
\begin{align}
&b \colon [0,+\infty) \times \cD_E \times \R \to \R,\qquad
&\sigma \colon [0,+\infty) \times \R \to (0,+\infty), \label{eq:bsigma}
\\
&K^i \colon [0,+\infty) \times \cD_E \times \R \times E \to \R,\qquad
&K^p \colon [0,+\infty) \times \cD_E \times \R \times E \to \R, \label{eq:K}
\\
&G^i \colon [0,+\infty) \times \cD_E \times \R \times Z \to \R,\qquad
&G^p \colon [0,+\infty) \times \cD_E \times \R \times Z \to \R, \label{eq:G}
\end{align}
and consider the following SDE with $X$ given in \eqref{eq:Xrepr}
\begin{equation}\label{eq:Ysde}
\left\{
\begin{aligned}
	\dd Y_t &= b(t, X_{t \wedge \cdot}, Y_t) \, \dd t + \sigma(t, Y_t) \, \dd W_t \\
	&+ \int_E K^i(t, X_{t^-\wedge \cdot}, Y_{t^-}, e) \, m^i(\dd t \, \dd e) + \int_E K^p(t, X_{t^-\wedge \cdot}, Y_{t^-}, e) \, m^p(\dd t \, \dd e) \\
	& + \int_Z G^i(t, X_{t^-\wedge \cdot}, Y_{t^-}, z) \, n^i(\dd t \, \dd z) + \int_Z G^p(t, X_{t^-\wedge \cdot}, Y_{t^-}, z) \, n^p(\dd t \, \dd z), \\
	Y_0 &= y_0 \in \R,
\end{aligned}
\right.
\end{equation}
where the measures $m^i, m^p, n^i, n^p$ are defined as 
\begin{align}
m^i(\dd t \, \dd e) &\coloneqq \ind_{(D^m \setminus J^m) \times E}(t,e) \, m(\dd t \, \dd e),
&
m^p(\dd t \, \dd e) &\coloneqq \ind_{J^m \times E}(t,e) \, m(\dd t \, \dd e),
\label{eq:mimp} \\
n^i(\dd t \, \dd e) &\coloneqq \ind_{(D^n \setminus J^n) \times Z}(t,z) \, n(\dd t \, \dd z),
&
n^p(\dd t \, \dd z) &\coloneqq \ind_{J^n \times Z}(t,z) \, n(\dd t \, \dd z),\label{ninp}
\end{align}
with
\begin{align}
	D^m &\coloneqq \{(\omega, t) \in \Omega \times (0,+\infty) \colon m(\omega; \, \{t\} \times E) = 1\}, \label{eq:Dm}\\
	J^m &\coloneqq \{(\omega, t) \in \Omega \times (0,+\infty) \colon \mu(\omega; \, \{t\} \times E) > 0\}, \label{eq:Jm} \\
	D^n &\coloneqq \{(\omega, t) \in \Omega \times (0,+\infty) \colon n(\omega; \, \{t\} \times Z) = 1\}, \label{eq:Dn}\\
	J^n &\coloneqq \{(\omega, t) \in \Omega \times (0,+\infty) \colon \nu(\omega; \, \{t\} \times Z) > 0\}. \label{eq:Jn}
\end{align}

\begin{assumption}\label{hp:jointlaw}
The partially observed system $(X,Y)$ satisfies the following properties.
\begin{enumerate}[(i)]
\item\label{hp:Yunique} $Y$ is the unique solution to SDE \eqref{eq:Ysde}.
\item \label{hp:m_n_disjoint} Random measures $m$ and $n$ do not have common jump times.
\item\label{hp:predproj} The $\bbF$-dual predictable projections $\mu$ and $\nu$ of random counting measures $m$ and $n$, respectively, are of the form
\begin{align}
	\mu((0,t] \times A) &= \int_0^t \int_A  Q^m(s,X_{s^- \wedge \cdot}; \dd e) \, \lambda^m(s, X_{s^- \wedge \cdot}) \, \dd s
	\\
	& + \int_0^t \int_A R^m(s,X_{s^- \wedge \cdot}; \dd e) \, \dd p^m_s, \quad t > 0, \, A \in \cB(E), \label{eq:mu} \\
	\nu((0,t] \times B) &= \int_0^t \int_B  Q^n(s, X_{s^-\wedge \cdot}, Y_{s^-}; \dd z)\, \lambda^n(s, X_{s^-\wedge \cdot}, Y_{s^-}) \, \dd s\notag\\
	& +\int_0^t \int_B R^n(s, X_{s^-\wedge \cdot}, Y_{s^-}; \dd z) \, \dd p^n_s, \quad t > 0, \, B \in \cZ,\label{eq:nu}
\end{align}
where
\begin{align}
	&Q^m \colon (0,+\infty) \times \cD_E \to \cP(E), \quad Q^n \colon (0,+\infty) \times \cD_E \times \R \to \cP(Z), \label{eq:Q} \\
	&R^m \colon (0,+\infty) \times \cD_E \to \cP(E), \quad R^n \colon (0,+\infty) \times \cD_E \times \R \to \cP(Z), \label{eq:R}
\end{align}
are probability transition kernels,
\begin{align}
\lambda^m \colon (0,+\infty) \times \cD_E \to [0,+\infty),\qquad
\lambda^n \colon (0,+\infty) \times \cD_E \times \R \to [0,+\infty) \label{eq:lambda}
\end{align}
are measurable functions verifying
\begin{align}
	&\sup_{(t, x) \in (0,+\infty) \times \cD_E} \lambda^m(t,x) < +\infty,\label{eq:lambdambdd} \\
	&\sup_{(t, x, y) \in (0,+\infty) \times \cD_E \times \R} \lambda^n(t,x,y) < +\infty\label{eq:lambdanbdd},
\end{align}
and
 $p^m \coloneqq (p^m_t)_{t \geq 0}$ and $p^n \coloneqq (p^n_t)_{t \geq 0}$ are $\bbY$-predictable counting processes
 satisfying, for all $t \geq 0$,
\begin{equation}\label{eq:ptmbdd}
\bbE[p^m_t] < +\infty, \qquad \bbE[p^n_t] < +\infty.
\end{equation}

\end{enumerate}
\end{assumption}




Whenever possible, to ease the notation, we use the following abbreviation for the coefficients of SDE~\eqref{eq:Ysde}. We set, for any $t \geq 0$,
\begin{align*}
	b_t &\coloneqq b(t, X_{t \wedge \cdot}, Y_t),
	&
	\sigma_t &\coloneqq \sigma(t, Y_t),
	\\
	K^i_t(\cdot) &\coloneqq K^i(t, X_{t^-\wedge \cdot}, Y_{t^-}, \cdot),
	& 	
	K^p_t(\cdot) &\coloneqq K^p(t, X_{t^-\wedge \cdot}, Y_{t^-}, \cdot),	
	\\
	G^i_t(\cdot) &\coloneqq G^i(t, X_{t^-\wedge \cdot}, Y_{t^-}, \cdot),
	& 	
	G^p_t(\cdot) &\coloneqq G^p(t, X_{t^-\wedge \cdot}, Y_{t^-}, \cdot),
\end{align*}
and therefore we can write SDE~\eqref{eq:Ysde} as
\begin{equation*}
\left\{
\begin{aligned}
	\dd Y_t &= b_t \, \dd t + \sigma_t \, \dd W_t + \int_E K^i_t(e) \, m^i(\dd t \, \dd e) + \int_E K^p_t(e) \, m^p(\dd t \, \dd e) \\
	& + \int_Z G^i_t(z) \, n^i(\dd t \, \dd z) + \int_Z G^p_t(z) \, n^p(\dd t \, \dd z), \\
	Y_0 &= y_0 \in \R.
\end{aligned}
\right.
\end{equation*}
To ensure that Assumption~\ref{hp:jointlaw}-(i) is verified, i.e. that SDE~\eqref{eq:Ysde} has a unique solution, one can consider several families of conditions. The interested reader may consult, for instance, \citep{jacod:calcul, oksendal:stochastic, protter:sde}.
Instead, to guarantee integrability, we ask the following.

\normalcolor

\begin{assumption}\label{hp:sdeexistuniq}
\mbox{}
Coefficients $b$, $\sigma$, $K^i$, $K^p$, $G^i$, $G^p$, given in~\eqref{eq:bsigma}, \eqref{eq:K}, \eqref{eq:G}
 satisfy, for all $t \geq 0$,
\begin{gather}
	\bbE\biggl[\int_0^t \abs{b_s} \, \dd s\biggr] < +\infty, \qquad \bbE\biggl[\int_0^t \sigma_s^2 \, \dd s\biggr] < +\infty,
	\\
	\bbE\biggl[\int_0^t \!\!\int_E \abs{K^i_s(e)} \, Q^m(s, X_{s^-\wedge \cdot}; \dd e)\dd s + \int_0^t \!\!\int_E \abs{K^p_s(e)} \, R^m(s, X_{s^-\wedge \cdot}; \dd e) \dd p^m_s\biggr] < +\infty,\\
\bbE\biggl[\int_0^t \!\!\int_Z \abs{G^i_s(z)} \, Q^n(s, X_{s^-\wedge \cdot}, Y_{s^-}; \dd z) \dd s + \int_0^t \!\!\int_Z \abs{G^p_s(z)} \, R^n(s, X_{s^-\wedge \cdot}, Y_{s^-}; \dd z) \dd p^n_s\biggr] < +\infty.
\end{gather}
\end{assumption}

\normalcolor

\begin{remark}
\label{rem:m_noexplosion}
\mbox{}
\begin{enumerate}
\item Conditions \eqref{eq:lambdambdd}, \eqref{eq:lambdanbdd}, and~\eqref{eq:ptmbdd} guarantee that random measures $m$ and $n$ are $\bbP$-a.s. non-explosive (see, e.g., \citep[Prop. 24.6]{davis:markovmodels}), i.e., jump times of $m$ and $n$ do not have a finite accumulation point. Said otherwise, in each compact time-interval $[0,t]$, for any $t \geq 0$, random measures $m$ and $n$ have only a finite number of jumps. 
\item 
\color{blue} Random sets $D^m$ and  $D^n$  in  \eqref{eq:Dm},  \eqref{eq:Dn}  are the supports of $m$ and $n$. \normalcolor 
Moreover  $J^m$ and  $J^n$ in \eqref{eq:Jm}, \eqref{eq:Jn} are the $\bbF$-predictable supports of $D^m$ and $D^n$, respectively, see Definition 2.32, Chapter I and Proposition 1.14, Chapter II, in \cite{jacod2013:limit}.
Under \eqref{eq:mu} and \eqref{eq:nu} the sets $J^m$, $J^n$, correspond to
\begin{align}
J^m  &= \{(\omega, t) \in \Omega \times (0,+\infty) \colon p^m_t(\omega)-p^m_{t-}(\omega) =1\}, \label{eq:Jm_bis} \\
J^n  &= \{(\omega, t) \in \Omega \times (0,+\infty) \colon p^n_t(\omega)-p^n_{t-}(\omega) =1\}. \label{eq:Jn_bis}
\end{align}
Then, standard results on random measures (see, e.g. Theorem 11.14
in \citet{he:semimart}) ensure that, in this case, $J^m \subset D^m$ and $J^n \subset D^n$. 
\qed
\end{enumerate}
\end{remark}

\begin{remark}\label{R:2.2}
 
\color{blue}
The process $Y$ is a fairly general jump-diffusion with some interesting features. First of all, its discontinuous part
accounts for jump times  that are in common with $X$,
 triggered by the random measure $m$,
 and additional jumps, described via the random measure $n$, that are disjoint from those of $X$, see Assumption \ref{hp:jointlaw}-\eqref{hp:m_n_disjoint}.
Second, random measures $m^i$, $n^i$ (resp. $m^p$,  $n^p$)  in SDE~\eqref{eq:Ysde} 
separate the contributions  
of  totally inaccessible (resp. $\bbY$-predictable) jump times of the process $Y$.
 In particular, $m^p$ (resp. $n^p$) accounts for $\bbY$-predictable jumps of $m$ (resp. $n$), while  $m^i$ (resp. $n^i$) accounts for totally inaccessible jumps of $m$ (resp. $n$). \qed
\end{remark}
\normalcolor

\begin{remark}
The dynamics of $Y$ given in~\eqref{eq:Ysde}, can also be written in a more compact notation as
\begin{equation}\label{eq:Ysde2}
	\left\{
	\begin{aligned}
		&\dd Y_t =b_t \, \dd t + \sigma_t \, \dd W_t + \int_\R y \, m^Y(\dd t \, \dd y),	& &t \geq 0, \\
		&Y_0 = y_0,
	\end{aligned}
	\right.
\end{equation}
where  $m^Y$  is the random counting measure associated to all jumps of $Y$,  defined as
\begin{equation}\label{mY}
m^Y\bigl((0,t] \times B\bigr) = \displaystyle \sum_{0 < s \leq t \colon \Delta Y_s \neq 0} \ind_{\Delta Y_s \in B} \quad \textup{for all}\,\, t \geq 0, \, B \in \cB(\R).
\end{equation}
It is worth noting that, for all $t \geq 0$ and any measurable function $g \colon \R \to \R$, it holds that 
\begin{multline}\label{eq:gintegral}
\int_0^t \int_\R g(y) \, m^Y(\dd t \, \dd y)
\\
= \int_0^t \int_E \ind_{K^i_s(e) \neq 0} \, g\bigl(K^i_s(e)\bigr) \, m^i(\dd s \, \dd e) + \int_0^t \int_E \ind_{K^p_s(e) \neq 0} \, g\bigl(K^p_s(e)\bigr) \, m^p(\dd s \, \dd e)
\\
+ \int_0^t \int_Z \ind_{G^i_s(z) \neq 0} \, g\bigl(G^i_s(z)\bigr) \, n^i(\dd s \, \dd z) + \int_0^t \int_Z \ind_{G^p_s(z) \neq 0} \, g\bigl(G^p_s(z)\bigr) \, n^p(\dd s \, \dd z),
\end{multline}
{\color{blue}
where $\ind_{K^i_s(e) \neq 0}$ (resp. $\ind_{K^p_s(e) \neq 0}$, $\ind_{G^i_s(z) \neq 0}$, $\ind_{G^p_s(z) \neq 0}$) denotes the indicator function of the (random) set $\{e \in E \colon K^i(s, X_{s \land \cdot}, Y_s, e) \neq 0\}$ (resp. $\{e \in E \colon K^p(s, X_{s \land \cdot}, Y_s, e) \neq 0\}$, $\{z \in Z \colon G^i(s, X_{s \land \cdot}, Y_s, z) \neq 0\}$, $\{z \in Z \colon G^p(s, X_{s \land \cdot}, Y_s, z) \neq 0\}$).
}
\qed
\end{remark}

\subsection{Comments}\label{sec:comments}
\color{blue}
We now make a brief overview of the characteristics and the main points of novelty of this work.
\normalcolor

In our setting, path-dependence with respect to the signal process is allowed both in the coefficients of SDE~\eqref{eq:Ysde} and in the local characteristics of random measures $m$ and $n$, as specified by their $\bbF$-dual predictable projections $\mu$ and $\nu$, given in~\eqref{eq:mu} and~\eqref{eq:nu}.
This represents an interesting feature from the point of view of applications, since it allows to model phenomena in which the behavior of the system shows a dependence on the past of the signal process. Such situations arise, for instance, due to memory effects or delays. From a theoretical point of view including path-dependence raises challenging issues, since classic filtering tools, used in Markovian frameworks, are either not suited or not developed to treat non-Markov situations.

Another key feature of our model is the presence of the $\bbY$-predictable counting processes $p^m$ and $p^n$ in the structure of the compensators $\mu$ and $\nu$, respectively. To the best of our knowledge, this case has not been analyzed in the literature so far. Existing filtering results of jump processes deal with the case where all jump times  of  $X$ and $Y$ are totally inaccessible, i.e., the case of a \emph{quasi-left continuous} compensator. Our setting, instead, covers also models where predictable jumps may occur in both  $X$ and $Y$.
\color{blue}
Counting processes $(p^m_t)_{t \geq 0}$  and $(p^n_t)_{t \geq 0}$, given in \eqref{eq:ptmbdd}, are assumed to be $\bbY$-predictable.
This is a key condition that permits to characterize explicitly the dual predictable projections of measures $m$ and $n$ with respect to the subfiltration $\bbY$. Indeed, when projecting on a smaller filtration, predictable jump times may change their nature.
\normalcolor This makes very difficult to compute the $\bbY$-compensators of measures  $(p^m_t)_{t \geq 0}$  and $(p^n_t)_{t \geq 0}$, unless they are already $\bbY$-predictable (see, e.g., the discussion after Theorem 2.1 in \citet{calzolaritorti:acctime} or Counterexample 4.8 in \citet{ditellajeanblanc:mrp}). }




\allowdisplaybreaks

\section{The Markovianization procedure}\label{sec:markov}
\color{blue}
In this section  we construct the \emph{history process} $\bfX = (\bfX_t)_{t \geq 0}$ through a \emph{Markovianization procedure}. The  process $\bfX$ is tightly linked with the signal $X$ and, most importantly, to its stopped trajectory $(X_{t \land \cdot})_{t \geq 0}$.

To  characterize the dynamics of the filter one typically follows two steps:  write, first,  the $\bbF$-semimartingale representation of the process $\bigl(\phi(X_{t \land \cdot})\bigr)_{t \geq 0}$, with  $\phi \colon \cD_E \to \R$ being  a bounded a measurable function; then,   derive its optional projection with respect to the observation filtration $\bbY$. 
\normalcolor
While, in principle, this is possible, thanks to the fact that the stopped process $(X_{t \land \cdot})_{t \geq 0}$ is a $\cD_E$-valued pure jump process (hence admitting a description through a random counting measure with some associated $\bbF$-compensator), in practice this is not an easy task because 
\begin{itemize}
\item[a)] One should be able to deduce the law of the stopped process, or equivalently to compute its $\bbF$-dual predictable projection, from the law of the signal process $X$ and this law should be expressed \emph{via} some probability transition kernel on $\cD_E$, which is a rather abstract object.
\item[b)] Models based on pure-jump processes are specified by heavily relying on the their underlying discrete time structure. This in turn means that jump intensity and a probability transition kernel depend on all possible sequences of past jump times and positions occupied by the pure-jump process itself.
\end{itemize}
\color{blue}
Introducing the history process allows to take these aspects into account.
\normalcolor 

We denote by $\delta$ some point outside $E$ and let $H$ be the state space of the history process, defined as the set of all sequences $h = (e_0, t_1, e_1, \dots)$, with $e_0 \in E$, $(e_n)_{n \in \N} \subset E \cup \{\delta\}$, and $(t_n)_{n \in \N} \subset (0,+\infty]$, satisfying, for all $n \in \N$,
\begin{enumerate}[1.]
\item $t_n \leq t_{n+1}$;
\item $t_n < +\infty \Longrightarrow t_n < t_{n+1}$;
\item $t_n = +\infty \Longleftrightarrow e_n = \delta$.
\end{enumerate}
We can introduce a metric on $H$ (see Appendix \ref{Section_Appendix_2}), that makes it a complete and separable metric space.

We define, also, the map $N \colon H \times (0,+\infty) \to \N_0$
\begin{equation}\label{eq:Nt}
N(h,t) = N(e_0, t_1, e_1, \dots, t) =
\begin{dcases}
\sup\{n \in \N \colon t_n < t\}, &\text{if } \{\cdots\} \neq \emptyset, \\
0,	&\text{otherwise,}
\end{dcases}
\quad h \in H, t \in (0,+\infty),
\end{equation}
that provides the number of jumps of a trajectory $h \in H$ up to time $t > 0$. Similarly, we define $N \colon H \to \N_0 \cup \{\infty\}$ as
\begin{equation}\label{eq:N}
N(h) = N(e_0, t_1, e_1, \dots) =
\begin{dcases}
\sup\{n \in \N \colon t_n < +\infty\}, &\text{if } \{\cdots\} \neq \emptyset, \\
0,	&\text{otherwise,}
\end{dcases}
\quad h \in H,
\end{equation}
which gives the total number of jumps of a trajectory $h \in H$.

Let us define for all $n \in \N$ the random sequences
\begin{align*}
	\xi_0 &\coloneqq (X_0, +\infty, \delta, +\infty, \delta, \dots);	&	\xi_n &\coloneqq (X_0, T_1, X_{T_1}, \dots, T_n, X_{T_n}, +\infty, \delta, +\infty, \delta, \dots),
\end{align*}
where $(T_n)_{n \in \N}$ is the sequence of jump times of $X$, given in~\eqref{eq:Xrepr}.
Finally, we set $\bfX \coloneqq (\bfX_t)_{t \geq 0}$ as
\begin{equation}
	\bfX_t \coloneqq
	\left\{
	\begin{aligned}
		&\xi_0,	&	&t \in [0, T_1), \\
		&\xi_n,	&	&t \in [T_n, T_{n+1}), \, n \in \N.
	\end{aligned}
	\right.
\end{equation}
This is a stochastic process taking values in $H$ whose purpose is to memorize all the past values and jump times of process $X$, hence the name \emph{history process}.
Moreover, it is a \mbox{pure-jump} process, with associated random counting measure
\begin{equation}\label{m}
	\bfm((0, t] \times B) \coloneqq \sum_{n \in \N} \ind_{T_n \leq t} \, \ind_{\xi_n \in B}, \quad t \geq 0, \, B \in \cB(H).
\end{equation}

We now reformulate our stochastic filtering problem by {\em replacing}, in a suitable way, the (unobservable) signal process with the history process. To do this, it is crucial to find a bijective correspondence between the stopped process $(X_{t \land \cdot})_{t \geq 0}$ and the history process $\bfX$. This is established in the next proposition.

\begin{proposition}\label{P_S}
Let $\widetilde \cD_E \subset \cD_E$ be the set of $E$-valued piecewise-constant c\`adl\`ag paths with no explosion in finite time.

There is one-to-one and onto correspondence between trajectories in  $\widetilde \cD_E$ and points in  $H$, namely there exists a bijective map
	 \begin{align}\label{S}
	 S \colon H \to \widetilde \cD_E.
	 \end{align}	
Moreover, $S$ and $S^{-1}$ are measurable and under~\eqref{eq:lambdambdd} and~\eqref{eq:ptmbdd} it holds that 
\begin{equation*}
X_{t \land \cdot} = S(\bfX_t), \quad \bfX_t = S^{-1}(X_{t \land \cdot}), \quad \bbP\text{-a.s.}, \, t \geq 0.
\end{equation*}
\end{proposition}
\begin{proof}
See Appendix \ref{Section_Appendix_2}.
\end{proof}

The map $S$ provided in~\eqref{S} is far from being an abstract object, since in several important cases it can be explicitly written. We give an example below.
\begin{example}\label{example_filter}
Let $x(\cdot)$ denote  the trajectory of an $E$-valued pure-jump process and consider the function $\phi \colon \widetilde \cD_E \rightarrow \R$ defined by
$$
\phi(x(\cdot \wedge t)) \coloneqq \sup_{s \in [0,\,t]} x(s), \quad t \geq 0.
$$
Clearly $y_t(\cdot) \coloneqq x(t \wedge \cdot) \in \widetilde \cD_E$, and, being $t \geq 0$ fixed, it has a finite number of jumps, say $N$.
Then $S^{-1}(y_t)= h$, for some $h \in H$ (depending on $t$), with $N(h) = N < \infty$.
In particular,
$\phi(y)= f \circ S^{-1}(y)$
with $f \colon H  \rightarrow \R$    given by
$$
f(h) = f(e_0, t_1, e_1, \dots) = \sup_{k \leq N(h)} e_k, \quad h \in H.
$$
\end{example}


\subsection{The Markovianized model}\label{subsec:markov}
From now on we consider a novel {\em partially observable system} $(\bfX, Y)$, which consists of the history (signal) process and the observed process. We call this pair the \emph{Markovianized model}.
{\color{blue} The aim of this section is to derive its dynamics.}

\medskip

\textsc{The history process.} To characterize the distribution of the history process $\bfX$ we compute the $\bbF$-dual predictable projection $\bfmu$ of the random counting measure $\bfm$, given in~\eqref{m}. The random measure $\bfmu$ is related to the $\bbF$-compensator $\mu$ (defined
in \eqref{eq:mu}) of the random counting measure $m$, associated to the jumps of the signal process $X$, as explained below.

Let $\Gamma \colon H \to E$ be the map given by
\begin{equation}\label{Gamma}
\Gamma(h) =
\begin{dcases}
e_{N(h)},	&\text{if } N(h) < \infty, \\
e,			&\text{if } N(h) = \infty,
\end{dcases}
\end{equation}
where $N$ is the map given in \eqref{eq:N} and $e \in E$ is any arbitrarily chosen point. 
It holds that $\Gamma(\bfX_t)=X_t, \, \bbP$-a.s., for any $t \geq 0$, that is the map $\Gamma$ in \eqref{Gamma} allows to recover from $\bfX$ the current values of the signal process $X$. Note that $\Gamma$ is measurable, onto but not one-to-one.

We also define maps $J_H$ and $J_{\widetilde \cD_E}$ that join trajectories in $H$ and $\widetilde \cD_E$, respectively, in the following way. $J_H \coloneqq H \times (0,+\infty) \times E \to H$ is defined for each $(h',t,e) \in H \times (0,+\infty) \times E$ as:
\begin{equation}\label{eq:JH}
J_H(h',t,e) \coloneqq (e_0, t_1, e_1, \dots) =
\begin{dcases}
e_0 \coloneqq e_0', \\
t_k \coloneqq t_k', \, e_k = e_k', 	&\text{if } k \neq N(h',t) + 1, \\
t_k \coloneqq t, \, e_k = e, 	&\text{if } k = N(h',t) + 1. \\
\end{dcases}
\end{equation}

On the other hand, $J_{\widetilde \cD_E} \coloneqq \widetilde \cD_E \times (0,+\infty) \times E \to \widetilde \cD_E$ is given for all $(x,t,e) \in \widetilde \cD_E \times (0,+\infty) \times E$ by:
\begin{equation}\label{eq:JDE}
J_{\widetilde \cD_E}(x,t,e) \coloneqq x(\cdot) \ind_{[0,t)}(\cdot) + e \ind_{[t,+\infty)}(\cdot).
\end{equation}

The action of map $J_H$ is to add to a trajectory $h' \in H$, with $N(h',t)$ jumps up to time $t > 0$,
a new jump (after the last one) at time $t$ to position $e \in E$. The resulting trajectory has $N(h',t) + 1$ jumps.

We have the following equalities:
\begin{align}
J_H(h',t,e) &= S^{-1}\left(J_{\widetilde \cD_E}(S(h'), t,e) \right), \quad \text{for all } (h',t,e) \in H \times (0,+\infty) \times E, \label{eq:JHJDE}\\
J_{\widetilde \cD_E}(x,t,e) &= S\left(J_H(S^{-1}(x), t,e) \right), \quad \text{for all } (x,t,e) \in \widetilde \cD_E \times (0,+\infty) \times E. \label{eq:JDEJH}
\end{align}

It can be proved that the map $J_H$ is continuous on $H \times (0,+\infty) \times E$ under the metric on $H$ introduced in Appendix \ref{Section_Appendix_2}, and therefore, measurable with respect to the product $\sigma$-algebra. Hence, thanks to \eqref{eq:JDEJH} and Proposition \ref{P_S}, the map $J_{\widetilde \cD_E}$ is measurable as well.%

We now define an $\bbF$-predictable random measure $\bfmu$ on $(0,+\infty) \times H$:
\begin{align}
	\bfmu((0,t] \times B) &\coloneqq \int_0^t \int_B \bflambda^m(s, \bfX_{s^-})\, \bfQ^m(s,\bfX_{s^-}; \dd h) \, \dd s\\
	& + \int_0^t \int_B \bfR^m(s,\bfX_{s^-}; \dd h) \,\dd p^m_s, \quad t > 0, \, B \in \cB(H), \label{eq:bfmu}
\end{align}
where $\bfQ^m \colon (0,+\infty) \times H \to \cP(H)$, $\bfR^m \colon (0,+\infty) \times H \to \cP(H)$  are the probability transition kernels given, for all $t > 0$, $H \in H$, and $B \in \cB(H)$, by
\begin{align}
\bfQ^m(t, h; \, B) &\coloneqq \ind_B\left(J_H(h,t,e)\right) Q^m(t, S(h); \, \dd e), \label{D:bfQ} \\
\bfR^m(t, h; \, B) &\coloneqq \ind_B\left(J_H(h,t,e)\right) R^m(t, S(h); \, \dd e), \label{D:bfR}
\end{align}%
and
$\bflambda^m \colon (0,+\infty) \times H \to (0,+\infty)$ is defined as
\begin{equation}\label{D:bflambda}
\bflambda^m(t, h) \coloneqq \lambda^m\bigl(t,S(h)\bigr). 
\end{equation}

We have the following important result.

\begin{proposition}\label{prop:bfmudualpredproj}
	Under Assumption~\ref{hp:jointlaw}, the random measure $\bfmu$ defined in~\eqref{eq:bfmu} is the $\bbF$-dual predictable projection of the random counting measure $\bfm$ in~\eqref{m}.
\end{proposition}
\begin{proof}
We show that the equality
\begin{equation*}
	\bbE \biggl[\int_0^t \int_H C_s(h) \, \bfm(\dd s \, \dd h) \biggr] = \bbE \biggl[\int_0^t \int_H C_s(h) \, \bfmu(\dd s \, \dd h)\biggr], \quad t \geq 0
\end{equation*}	
holds for any non-negative $\bbF$-predictable random field $C \colon \Omega \times [0,+\infty) \times H \to \R$. By a standard monotone class argument, it is enough to prove the equality above for $C_t(h) = \gamma_t \ind_B(h)$, where $\gamma = (\gamma)_{t \geq 0}$ is a non-negative $\bbF$-predictable random process and $B \in \cB(H)$.

The key property is that $\bfX$ satisfies $\bfX_{T_n} = J_H(\bfX_{T_n^{-}},T_n, X_{T_n})$, for each $n \in \N$. Recalling also that $X_{t \land \cdot} = S(\bfX_t)$, $\bbP$-a.s., for all $t \geq 0$, we easily obtain:
	\begin{align*}
		&\bbE \biggl[\int_0^t\!\! \int_H C_s(h) \, \bfm(\dd s \, \dd h)\biggr] =
		\bbE \biggl[\sum_{n=1}^\infty \ind_{T_n \leq t} \,\gamma_{T_n} \ind_B(\bfX_{T_n})\biggr]
		\\
		= \, &\bbE \biggl[\sum_{n=1}^\infty \ind_{T_n \leq t} \,\gamma_{T_n} \ind_B(J_H(\bfX_{T_n^{-}},T_n, X_{T_n}))\biggr]
		\\
		= \, &\bbE \biggl[\int_0^t \!\!\int_E \gamma_s \ind_B(J_H(\bfX_{s^{-}},s, e)) \, m(\dd s \, \dd e)\biggr] =
		\bbE \biggl[\int_0^t \!\!\int_E \gamma_s \ind_B(J_H(\bfX_{s^{-}},s, e)) \, \mu(\dd s \, \dd e)\biggr]
		\\
		= \, &\bbE \biggl[\int_0^t \!\! \int_H \gamma_s \ind_B(h) \,\bflambda^m(s, \bfX_{s^-}) \, \bfQ^m(s, \bfX_{s^-}; \dd h)) \, \dd s\biggr]\\
& + \bbE \biggl[\int_0^t \!\!\int_H \gamma_s \ind_B(h) \, \bfR^m(s, \bfX_{s^-}; \dd h)) \, \dd p^m\biggr] \\
		= \, & \bbE \biggl[\int_0^t \!\!\int_H C_s(h) \, \bfmu(\dd s \, \dd h)\biggr]. \qedhere
	\end{align*}
\end{proof}

\begin{remark}\label{rem:Dm}
It is important to stress that the following equalities hold:
\begin{align*}
	\{(\omega, t) \in \Omega \times (0,+\infty) \colon \bfm(\omega; \, \{t\} \times H) = 1\}&=D^m, \\
	\{(\omega, t) \in \Omega \times (0,+\infty) \colon \bfmu(\omega; \, \{t\} \times H) > 0\}&=J^m,
\end{align*}
where $D^m$ and $J^m$ are the random sets defined in~\eqref{eq:Dm} and~\eqref{eq:Jm}. This implies, in particular, that Remark \ref{rem:m_noexplosion} remains valid  for the history process $\bfX$.
\end{remark}

\textsc{The observed process.} For $a=i,p$ and any $h, h' \in H$, $y,z \in \R$, $t \geq 0$ and $B \in \cB(Z)$, we set
\begin{equation}\label{eq:obscoeffs}
\begin{aligned}
\bfb(t,h, y)	&\coloneqq b(t,S(h), y),
&
\bfG^a(t,h, y, z)	&\coloneqq G^a\bigl(t,S(h), y, z\bigr), \\
\bfK^a(t,h, y, h')	&\coloneqq K^a\bigl(t,S(h), y, \Gamma(h')\bigr),
&
\bfQ^n(t, h, y; \, B) &\coloneqq Q^n\bigl(t, S(h),y; \, B \bigr), \\	
\bfR^n(t, h,y; \, B) &\coloneqq R^n\bigl(t, S(h),y; \, B\bigr),
&
\bflambda^n(t, h, y) &\coloneqq \lambda^n\bigl(t,S(h),y\bigr),
\end{aligned}
\end{equation}
where $S$ is the map given in Proposition~\ref{P_S} and $\Gamma$ is defined in~\eqref{Gamma}. 
Sometimes, to ease the notation, we write
\begin{align*}
	\bfb_t &\coloneqq \bfb(t,\bfX_t, Y_t),
	&
 \bfG^a_t(\cdot) &\coloneqq \bfG^a(t, \bfX_{t^-}, Y_{t^-}, \cdot), &\bfK^a_t(\cdot) &\coloneqq \bfK^a(t, \bfX_{t^-}, Y_{t^-}, \cdot),
\end{align*}
for any $t \geq 0$ and $a = i,p$.

Using Equation \eqref{eq:Ysde2} we can write the dynamics of $Y$ in terms of the history process as
\begin{equation}\label{eq:YSDE}
	\left\{
	\begin{aligned}
		\dd Y_t &= \bfb(t, \bfX_t, Y_t) \, \dd t + \sigma(t, Y_t) \, \dd W_t + \int_\R y \, m^Y(\dd t \, \dd y)\\
		Y_0 &= y_0 \in \R,
	\end{aligned}
	\right.
\end{equation}

It is useful to introduce the following sets, for $a=i,p$, $t \geq 0$, $h' \in H$, $y \in \R$, $B \in \cB(\R)$:
\begin{align}
d^{a,K}(B,t,h', y) &\coloneqq \{h \in H \colon \bfK^a(t, h', y, h)\in B\setminus\{0\}\},
&
D^{a,K}_t(B) &\coloneqq  d^{a,K}(B,t, \bfX_{t^-}, Y_{t^-}), \label{eq:DKt}
\\
d^{a,G}(B,t,h', y) &\coloneqq \{z \in Z \colon \bfG^a(t, h', y, z) \in B\setminus\{0\}\},
&
D^{a,G}_t(B) &\coloneqq  d^{a,G}(B,t, \bfX_{t^-}, Y_{t^-}). \label{eq:DGt}
\end{align}
We omit the dependence on the set if $B=\R$.

\begin{remark}
Thanks to the sets introduced above and the coefficients $\bfG^i$, $\bfG^p$, $\bfK^i$, $\bfK^p$ defined in~\eqref{eq:obscoeffs}, we can rewrite~\eqref{eq:gintegral} as
\begin{multline}
\int_0^t\int_\R g(y) \, m^Y(\dd s \, \dd y)
\\
= \int_0^t \int_H \ind_{D^{i,K}_s}(h) \, g\bigl(\bfK^i_s(h)\bigr) \, \bfm^i(\dd s \, \dd h)
+ \int_0^t \int_Z \ind_{D^{i,G}_s}(z) \, g\bigl(\bfG^i_s(z)\bigr) \, n^i(\dd s \, \dd z)
\\
+ \int_0^t \int_H \ind_{D^{p,K}_s}(h) \, g\bigl(\bfK^p_s(h)\bigr) \, \bfm^p(\dd s \, \dd h)
+ \int_0^t \int_Z \ind_{D^{p,G}_s}(z) \, g\bigl(\bfG^p_s(z)\bigr) \, n^p(\dd s \, \dd z),
\end{multline}
for all $t \geq 0$ and for any measurable function $g \colon \R \to \R$.
Here, similarly to Section~\ref{sec:model}, we define random measures $\bfm^i$ and $\bfm^p$ as:
\begin{align}
\qquad \bfm^i(\dd t \, \dd h) &\coloneqq \ind_{(D^m \setminus J^m) \times H}(t,h) \, \bfm(\dd t \, \dd h),
&
\bfm^p(\dd t \, \dd h) &\coloneqq \ind_{J^m \times H}(t,h) \, \bfm(\dd t \, \dd h).  \quad \qed
\end{align}
\end{remark}

It is clear that under Assumptions~\ref{hp:jointlaw}-\eqref{hp:m_n_disjoint} the jump times of $\bfm$ and $n$ are $\bbP$-a.s. disjoint (see also Remark~\ref{rem:Dm}). Moreover, under Assumption~\ref{hp:sdeexistuniq}, the functions $\bfb$, $\bfG^i$, $\bfG^p$, $\bfK^i$, $\bfK^p$, appearing in~\eqref{eq:obscoeffs} are measurable and satisfy, for all $t \geq 0$, $\displaystyle \bbE\biggl[\int_0^t \abs{\bfb_s} \, \dd s\biggr] < +\infty$ and
\begin{align*}
&\bbE\biggl[\int_0^t\!\! \int_H \abs{\bfK^i_s(h)} \, \bfQ^m(s, \bfX_{s^-}; \dd h) \dd s + \int_0^t\!\! \int_H \abs{\bfK^p_s(h)} \, \bfR^m(s, \bfX_{s^-}; \dd h)\dd p^m_s\biggr] < +\infty,\\
&\bbE\biggl[\int_0^t \int_Z \abs{\bfG^i_s(z)} \, \bfQ^n(s, \bfX_{s^-}, Y_{s^-}; \dd z) \dd s + \int_0^t \int_Z \abs{\bfG^p_s(z)} \, \bfR^n(s, \bfX_{s^-}, Y_{s^-}; \dd z)\dd p^n_s\biggr] < +\infty.\end{align*}

Finally, we compute the $\bbF$-dual predictable projection of the random counting measure $m^Y$. The proof is based on standard arguments and given in Appendix~\ref{Section_Appendix_5}.

\begin{proposition}\label{prop:mYFcomp}
  Under Assumptions \ref{hp:jointlaw}, \ref{hp:sdeexistuniq}, the $\bbF$-dual predictable projection of the random counting measure $m^Y$ is given by:
  \begin{align}
    \bfmu^Y\bigl((0,t] \times B\bigr) &= \int_0^t \biggl\{ \int_H \ind_{D_s^{i,K}(B)}(h) \, \bflambda^m(s, \bfX_{s^-}) \, \bfQ^m(s, \bfX_{s^-} ; \dd h) \\
    &\qquad + \int_Z \ind_{D_s^{i,G}(B)}(z) \, \bflambda^n(s, \bfX_{s^-}, Y_{s-}) \, \bfQ^n(s, \bfX_{s^-} , Y_{s-}; \dd z) \biggr\} \, \dd s \\
    &+ \int_0^t \int_H \ind_{D_s^{p,K}(B)}(h) \, \,\bfR^m(s, \bfX_{s^-}; \dd h) \,\dd p_s^m \\
    &+ \int_0^t \int_Z \ind_{D_s^{p,G}(B)}(z) \, \,\bfR^n(s, \bfX_{s^-}, Y_{s-}; \dd z) \,\dd p_s^n. \label{eq:mYFcomp}
  \end{align}
for all $t > 0, \, B \in \cB(\R)$.
\end{proposition}

{\color{blue}
\section{The filter}\label{sec:filtprob}
}
We introduce the \emph{filter} $\pi = (\pi_t)_{t \geq 0}$, as 
\begin{equation}\label{eq:pidef}
	\pi_t(\phi) = \bbE[\phi(X_{t \land \cdot}) \mid \cY_t], \quad t \geq 0,
\end{equation}
for any bounded and measurable function $\phi \colon \cD_E \to \R$.
{\color{blue}
Recall that $\cD_E$ denotes the set of $E$-valued c\`adl\`ag functions on $[0,+\infty)$ and the notation $X_{t \land \cdot}$ indicates the path of process $X$ stopped at time $t \geq 0$, i.e., $(X_{t \land \cdot})_{t \geq 0}$ is the $\cD_E$-valued stochastic process such that, for any $t \geq 0$, $X_{t \land \cdot} = \{s \mapsto X_{t \land s}\}_{s \geq 0}$.
}

Since $\cD_E$ is a complete and separable metric space, the process $\pi$ is well-defined, $\cP(\cD_E)$-valued and $\bbY$-adapted. Moreover, $\pi$ admits a c\`adl\`ag modification, since $X$ is c\`adl\`ag (see, e.g. \citep[Theorem~2.1]{bain:fundofstochfilt}).

The function $\phi$ in equation~\eqref{eq:pidef}
depends on the history of the signal process $X$ in a non-anticipative way, i.e., through the stopped trajectory $X_{t \land \cdot}$, $t \geq 0$. Considering this kind of dependence is one of the main features of our model and generalizes filtering results available in the literature. Our setting is characterized by the lack of Markovianity due to the fact that $X$ is not assumed to be Markovian and the dynamics of $Y$ depend on the whole trajectory of the signal. 
Existing results on filtering, instead, mainly cover the Markovian case, where the filter is completely characterized by computing the conditional law of each of the random variables $X_t$ given $\cY_t$, for all $t \geq 0$.

{\color{blue}To compute then filtering equation we make use the Markovianization procedure described in Section~\ref{sec:markov} and the relationship between $X$ and $\bfX$. Therefore, we first study the filtering problem for the partially observed system $(\bfX, Y)$, whose dynamics have been derived in Section~\ref{subsec:markov}.
We introduce} a new filter  $\Pi \coloneqq (\Pi_t)_{t \geq 0}$, satisfying
\begin{equation}\label{eq:Pidef}
\Pi_t(f) = \bbE\bigl[f(\bfX_t) \mid \cY_t\bigr], \quad t \geq 0,
\end{equation}
for any bounded and measurable function $f \colon H \to \R$%
{\color{blue}, where $H$ is the state space of the history process $\bfX$.
Also in this case, since $H$ is a complete and separable metric space, the process $\Pi$ is well-defined, $\cP(H)$-valued, $\bbY$-adapted, and admitting a c\`adl\`ag modification.}
{\color{blue} Finding a representation of $\Pi$ as a solution of a {\em filtering equation}
allows us deduce a filtering equation for $\pi$. These two results are proved in Section~\ref{sec:filter}}.

We remind that the map $S$ of Proposition~\ref{P_S} enables us to go back and forth between the original model of Section~\ref{sec:model} and the Markovianized one. In particular, there is a relation between the stochastic processes $\pi$ and $\Pi$, corresponding to the original model and the Markovianized model respectively, which is described in the following Lemma.

\begin{lemma}\label{L:4.1}
Let $\pi$ and $\Pi$ be the  processes defined in~\eqref{eq:pidef} and~\eqref{eq:Pidef}, respectively. Then, for any bounded and measurable function $\phi \colon \cD_E \to \R$, the function $f \colon H \to \R$ defined as $f \coloneqq \phi \circ S$ is bounded and measurable and
\begin{equation}\label{eq:filterequality}
	\Pi_t(f) = \pi_t(\phi),	\quad \bbP\text{-a.s.}, \, t \geq 0.
\end{equation}
Vice versa, for any bounded and measurable function $f \colon H \to \R$, the function $\phi \colon \cD_E \to \R$ given by $\phi \coloneqq f \circ S^{-1}$ is bounded and measurable and~\eqref{eq:filterequality} holds.
\end{lemma}
The proof of this Lemma is omitted since it is an immediate consequence of Proposition~\ref{P_S}.

\section{Martingale representation theorem under the observed filtration}\label{sec:martrepr}
In this section we consider the Markovianized model introduced in Section~\ref{subsec:markov}.
To characterize the dynamics of the filter $\Pi$, see equation~\eqref{eq:Pidef}, it is necessary to provide a representation theorem for martingales with respect to the observed filtration $\bbY$. The form of the observed process $Y$ suggests that $\bbY$-martingales can be represented as the sum of two stochastic integrals, respectively driven by a $\bbY$-Brownian motion and the $\bbY$-compensated jump measure of $Y$. Therefore, as a first step, we  compute the $\bbY$-dual predictable projection of $m^Y$, defined in~\eqref{mY}.

In the sequel we denote integrals with respect to the $\cP(H)$-valued random variable $\Pi_t, \, t \geq 0$, as
$$
\Pi_t\bigl(F(t,\cdot,Y_t)\bigr) \coloneqq \int_H F(t, h, Y_t)\, \Pi_t(\dd h),
$$
for any bounded and measurable function $F \colon [0,+\infty) \times H \times \R \to \R$, and
$$
\Pi_t\bigl(R(t, \cdot, Y_t \, ; \, \dd h)\bigr) \coloneqq \int_H R(t, h', Y_t \, ; \, \dd h)\, \Pi_t(\dd h'),
$$
for any transition kernel $R$ from $\bigl([0,+\infty) \times H \times \R\bigr)$ into $\bigl(H, \cB(H)\bigr)$.

\begin{proposition}\label{prop:mYYcomp}
  Under Assumptions \ref{hp:jointlaw} and \ref{hp:sdeexistuniq}, the $\bbY$-dual predictable projection of the random counting measure $m^Y$ is given by:
  \begin{align}
	\hat\bfmu^Y\bigl((0,t] \times B\bigr)
	&= \int_0^t \biggl\{\int_H \Pi_{s^-}\bigl(\ind_{d^{i,K}(B, s, \cdot, Y_{s^-})}(h) \,\bflambda^m(s, \cdot) \bfQ^m(s, \cdot ; \dd h)\bigr) \\
    &\qquad + \int_Z \Pi_{s^-}\bigl(\ind_{d^{i,G}(B, s, \cdot, Y_{s^-})}(z)  \,\bflambda^n(s, \cdot, Y_{s-}) \bfQ^n(s, \cdot, Y_{s-}; \dd z)\bigr) \biggr\} \, \dd s
	\\
	&+ \int_0^t \int_H \Pi_{s^-}\bigl(\ind_{d^{p,K}(B,s, \cdot, Y_{s^-})}(h) \,\bfR^m(s, \cdot; \dd h)\bigr) \, \dd p_s^m \\
	&+ \int_0^t \int_Z \Pi_{s^-}\bigl(\ind_{d^{p,G}(B,s, \cdot, Y_{s^-})}(z) \,\bfR^n(s, \cdot, Y_{s-}; \dd z)\bigr) \, \dd p_s^n, \label{eq:mYYcomp}
\end{align}
for all $t > 0, \, B \in \cB(\R)$.
\end{proposition}

\begin{proof}
We need to show that the equality
\begin{equation*}
	\bbE \biggl[\int_0^t \int_\R C_s(y) \, m^Y(\dd s \, \dd y)\biggr] = \bbE \biggl[\int_0^t \int_\R C_s(y) \, \hat \bfmu^Y(\dd s \, \dd y)\biggr],
\end{equation*}	
holds for any $t \geq 0$, and any non-negative and $\bbY$-predictable random field $C \colon \Omega \times [0,\infty) \times \R \to \R$.

It is sufficient to prove it for $C$ of the form $C_t(y) = \gamma_t \, \ind_B(y)$, where $\gamma$ is a non-negative $\bbY$-predictable process and $B \in \cB(\R)$. The conclusion follows, then, by a standard monotone class argument.

Since $\bfmu^Y$, given in~\eqref{eq:mYFcomp}, is the $\bbF$-compensator of the random counting measure $m^Y$ and $C$ is, in particular, $\bbF$-predictable, we get that, for all $t \geq 0$,

\begin{align}
	&\mathrel{\phantom{=}} \bbE \biggl[\int_0^t \gamma_s \int_\R \ind_B(y) \, m^Y(\dd s \, \dd y)\biggr] = \bbE \biggl[\int_0^t \gamma_s \int_\R \ind_B(y) \, \bfmu^Y(\dd s \, \dd y)\biggr] \\
	&= \bbE \biggl[\int_0^t \gamma_s \biggl\{ \int_H \ind_{d^{i,K}(B,s,\bfX_{s^-},Y_{s^-})}(h) \, \bflambda^m(s, \bfX_{s^-}) \, \bfQ^m(s, \bfX_{s^-} ; \dd h) \\
 	&\qquad + \int_Z \ind_{d^{i,G}(B,s,\bfX_{s^-},Y_{s^-})}(z) \, \bflambda^n(s, \bfX_{s^-}, Y_{s^-}) \, \bfQ^n(s, \bfX_{s^-} , Y_{s^-}; \dd z) \biggr\} \, \dd s\biggr] \\
	&+ \bbE \biggl[\int_0^t \gamma_s \int_H \ind_{d^{p,K}(B,s,\bfX_{s^-},Y_{s^-})}(h) \, \,\bfR^m(s, \bfX_{s^-}; \dd h) \,\dd p_s^m\biggr] \\
	&+ \bbE \biggl[\int_0^t \gamma_s \int_Z \ind_{d^{p,G}(B,s,\bfX_{s^-},Y_{s^-})}(z) \, \,\bfR^n(s, \bfX_{s^-}, Y_{s^-}; \dd z) \,\dd p_s^n\biggr]\label{eq:mYFcomp2}
\end{align}	

We can deal with the first two lines of this equality by repeatedly applying the \mbox{Fubini-Tonelli} Theorem and properties of conditional expectations, to get:
\begin{align*}
&\mathrel{\phantom{=}} \bbE \biggl[\int_0^t \gamma_s \biggl\{ \int_H \ind_{d^{i,K}(B,s,\bfX_{s^-},Y_{s^-})}(h) \, \bflambda^m(s, \bfX_{s^-}) \, \bfQ^m(s, \bfX_{s^-} ; \dd h) \\
&\qquad + \int_Z \ind_{d^{i,G}(B,s,\bfX_{s^-},Y_{s^-})}(z) \, \bflambda^n(s, \bfX_{s^-}, Y_{s^-}) \, \bfQ^n(s, \bfX_{s^-} , Y_{s^-}; \dd z) \biggr\} \, \dd s\biggr] \\
&= \int_0^t \bbE \biggl[\gamma_s \, \bbE \biggl[\int_H \ind_{d^{i,K}(B,s,\bfX_{s^-},Y_{s^-})}(h) \, \bflambda^m(s, \bfX_{s^-}) \, \bfQ^m(s, \bfX_{s^-} ; \dd h) \\
&\qquad + \int_Z \ind_{d^{i,G}(B,s,\bfX_{s^-},Y_{s^-})}(z) \, \bflambda^n(s, \bfX_{s^-}, Y_{s^-}) \, \bfQ^n(s, \bfX_{s^-} , Y_{s^-}; \dd z) \biggm| \cY_s \biggr] \, \dd s\biggr] \\
&=\bbE \biggl[ \int_0^t \gamma_s \biggl\{\int_H \Pi_{s^-}\bigl(\ind_{d^{i,K}(B, s, \cdot, Y_{s^-})}(h) \,\bflambda^m(s, \cdot) \bfQ^m(s, \cdot ; \dd h)\bigr) \\
&\qquad + \int_Z \Pi_{s^-}\bigl(\ind_{d^{i,G}(B, s, \cdot, Y_{s^-})}(z)  \,\bflambda^n(s, \cdot, Y_{s^-}) \bfQ^n(s, \cdot, Y_{s^-}; \dd z)\bigr) \biggr\} \, \dd s \biggr]. 	
\end{align*}

To elaborate the last two lines of~\eqref{eq:mYFcomp2} we observe, first, that for any $s \geq 0$
\begin{multline}
\gamma_s \int_H \Pi_{s^-}\bigl(\ind_{d^{p,K}(B,s, \cdot, Y_{s^-})}(h) \,\bfR^m(s, \cdot; \dd h)\bigr) \\
= \bbE \biggl[\gamma_s \, \int_H \ind_{d^{p,K}(B,s,\bfX_{s^-},Y_{s^-})}(h) \, \,\bfR^m(s, \bfX_{s^-}; \dd h) \biggm| \cY_{s^-} \biggr],
\end{multline}
since $\gamma$ is $\bbY$-predictable, hence $\gamma_s$ is $\cY_{s^-}$-measurable (see, e.g., \citep[Exercise 7.7.5]{cohen:stochcalculus}). Moreover, the process on the left hand side, being the product of a $\bbY$-predictable process and a left-continuous $\bbY$-adapted process, is $\bbY$-predictable, and then it is a version of the $\bbY$-predictable projection of
\begin{equation*}
\gamma_s \, \int_H \ind_{d^{p,K}(B,s,\bfX_{s^-},Y_{s^-})}(h) \, \,\bfR^m(s, \bfX_{s^-}; \dd h), \quad s \geq 0.
\end{equation*}
A similar reasoning holds also for the $\bbY$-predictable projection of
projection of
\begin{equation*}
\gamma_s \int_Z \ind_{d^{p,G}(B,s,\bfX_{s^-},Y_{s^-})}(z) \, \,\bfR^n(s, \bfX_{s^-}, Y_{s-}; \dd z), \quad s \geq 0.
\end{equation*}

In view of these facts, we can apply \citep[Ch. VI, Theorem 57]{dellacheriemeyer:B} to the last two lines of~\eqref{eq:mYFcomp2} and get:
\begin{align*}
&\mathrel{\phantom{=}} \bbE \biggl[\int_0^t \gamma_s \int_H \ind_{d^{p,K}(B,s,\bfX_{s^-},Y_{s^-})}(h) \, \,\bfR^m(s, \bfX_{s^-}; \dd h) \,\dd p_s^m\biggr] \\
&+ \bbE \biggl[\int_0^t \gamma_s \int_Z \ind_{d^{p,G}(B,s,\bfX_{s^-},Y_{s^-})}(z) \, \,\bfR^n(s, \bfX_{s^-}, Y_{s^-}; \dd z) \,\dd p_s^n\biggr] \\
&= \bbE \biggl[\int_0^t \gamma_s \, \int_H \Pi_{s^-}\bigl(\ind_{d^{p,K}(B,s, \cdot, Y_{s^-})}(h) \,\bfR^m(s, \cdot; \dd h)\bigr) \, \dd p_s^m \biggr] \\
&+ \bbE \biggl[\int_0^t \gamma_s \, \int_Z \Pi_{s^-}\bigl(\ind_{d^{p,G}(B,s, \cdot, Y_{s^-})}(z) \,\bfR^n(s, \cdot, Y_{s^-}; \dd z)\bigr) \, \dd p_s^n \biggr]. \qedhere
\end{align*}
\end{proof}

\subsection{The martingale representation theorem}\label{sec:martreprth}
In this section we provide a martingale representation theorem with respect to the filtration $\bbY$.

We consider a pair $(I, m^Y- \hat \bfmu^Y)$ consisting of a $\bbY$-Brownian motion $I \coloneqq (I_t)_{t \geq 0}$ and the $\bbY$-compensated random counting measure associated to the jumps of the observed process $Y$. We call this pair the {\em innovation}.

The process $I$ is defined as
\begin{equation}\label{Innov}
I_t \coloneqq W_t +   \int_0^t \frac{\bfb(s, \bfX_s, Y_{s}) - \Pi_t(\bfb(s, \cdot, Y_s))}{\sigma(s, Y_s)}   \,\dd s, \quad t \geq 0 .
\end{equation}

We shall work under the following hypothesis.
\begin{assumption}\label{ass:girsanov}
$$
\color{blue}\bbE\left[\exp\biggl\{ \frac{1}{2}\int_0^t\biggl(\frac{b_s}{\sigma_s}\biggr)^2 \dd s \biggr\}\right]
\normalcolor= \bbE\left[\exp\biggl\{ \frac{1}{2}\int_0^t\biggl(\frac{\bfb_s}{\sigma_s}\biggr)^2 \dd s \biggr\}\right]< \infty, \quad \text{for all } t \geq 0.
$$
\end{assumption}

This Novikov-type condition can be weakened, but we do not pursue any generalization here. For more details, see, e.g., \citep[Lemma 3.9]{bain:fundofstochfilt} or \citep{cecicolaneri:ks}.

\begin{lemma}
	Under Assumption~\ref{ass:girsanov}, $I$ is a $\bbY$-Brownian motion.
\end{lemma}
\begin{proof}
The result is shown, for instance, in  \citep[Theorem 7.17]{liptsershiryaev2001:statistics}.
\end{proof}

{\color{blue}%
Thanks to Assumption~\ref{ass:girsanov}, the process $L$ given by
\begin{equation}
L_t \coloneqq \mathcal E \left(-\int_0^t \frac{\Pi_t(\bfb(s, \cdot, Y_s))}{\sigma(s, Y_s)}  \,\dd I_s\right), \quad t \geq 0,
\end{equation}
where $\cE(\cdot)$ denotes the Dol\'eans-Dade exponential, is
a $\bbY$-martingale.
Therefore, for any fixed $T > 0$, we can define the probability measure $\widetilde \bbP$, equivalent to $\bbP$ on $(\Omega, \cF_T)$ by
\begin{equation*}
\frac{\dd \widetilde \bbP}{\dd \bbP} \biggr|_{\cF_T} = L_T.
\end{equation*}
By Girsanov's Theorem the process $\widetilde W$
\begin{equation}\label{eq:Wtilde}
\widetilde W_t  \coloneqq
I_t + \int_0^t \frac{\Pi_s(\bfb(s, \cdot, Y_s))}{\sigma(s, Y_s)} \dd s=
W_t + \int_0^t \frac{\bfb(s, \bfX_s, Y_s)}{\sigma(s, Y_s)} \dd s, \quad t \in [0,T],
\end{equation}
is a $((\cY_t)_{t \in [0,T]},\widetilde \bbP)$-Brownian motion.

We suppose that it is possible to represent local martingales under probability $\widetilde \bbP$ in the filtration generated by the Brownian motion $\widetilde W$ and by the jump measure $m^Y$. To be more precise, we introduce, for any $T > 0$,
\begin{align*}
&\bbG^T \coloneqq (\cG_t)_{t \in [0,T]},
&
&\cG_t \coloneqq \sigma\bigl(\widetilde W_s, \, 0 \leq s \leq t\bigr) \lor \cN, \\
&\bbH^T \coloneqq (\cH_t)_{t \in [0,T]},
&
&\cH_t \coloneqq \sigma\bigl(m^Y((0,s] \times B, \, 0 \leq s \leq t, \, B \in \cB(\R) \bigr) \lor \cN,
\\
&\bbG^T \vee \bbH^T \coloneqq (\cG_t \vee \cH_t)_{t \in [0,T]},
&
&\cG_t \vee \cH_t \coloneqq \sigma(\cG_t \cup \cH_t).
\end{align*}

\begin{assumption}\label{hp:martrepr}
We assume that for any $T > 0$ filtration $\bbG^T \vee \bbH^T$ is right-continuous and that any $(\bbG^T \vee \bbH^T, \widetilde \bbP)$-local martingale $\widetilde M = (\widetilde M_t)_{t \in [0,T]}$ admits the representation
\begin{align*}
\widetilde M_t= \widetilde M_0 + \int_0^t \widetilde \gamma_s \, \dd \widetilde W_s + \int_0^t \int_\R \widetilde U_s(y) (m^Y - \hat \bfmu^Y) (\dd s \,\dd y), \quad t \in [0,T],
\end{align*}
where $\widetilde \gamma$ and $\widetilde U$ are respectively a $\bbG^T \vee \bbH^T$-predictable process and a $\bbG^T \vee \bbH^T$-predictable random field, satisfying
\begin{equation}
\int_0^T |\widetilde \gamma_s|^2 \dd s < \infty, \quad
\int_0^T \int_\R |\widetilde U_s(y)| \, \hat \bfmu^Y(\dd s \, \dd y) < \infty, \quad \bbP\text{-a.s.}
\end{equation}
\end{assumption}
}

The main result of this section is given by the following theorem.

\begin{theorem}\label{T:mtg_rep}
{\color{blue}%
Suppose that 
Assumptions~\ref{hp:jointlaw}, \ref{hp:sdeexistuniq}, \ref{ass:girsanov}, and \ref{hp:martrepr} are verified.
Then, for any $T > 0$ filtration $(\cY_t)_{t \in [0,T]}$ is right-continuous and } every $(\cY_t)_{t \in [0,T]}$-local martingale $M$ admits the decomposition
\begin{align}\label{eq:martrepr}
M_t= M_0 + \int_0^t \gamma_s \, \dd I_s + \int_0^t \int_\R U_s(y) (m^Y - \hat \bfmu^Y) (\dd s \,\dd y), \quad t \in [0,T],
\end{align}
where $\gamma$ and $U$ are respectively a $(\cY_t)_{t \in [0,T]}$-progressive process and a $(\cY_t)_{t \in [0,T]}$-predictable random field, satisfying
\begin{equation}\label{int_cond}
\int_0^T |\gamma_t|^2 \, \dd t < \infty, \quad
\int_0^T \int_\R |U_t(y)| \, \hat \bfmu^Y(\dd t \, \dd y) < \infty, \quad \bbP\text{-a.s.}
\end{equation}

Moreover, $U$ can be chosen so that the process $\hat U$, defined by
\begin{equation}\label{eq:hatU}
\hat U_t \coloneqq \int_\R U_t(y) \, \ind_{t \leq T} \, \hat \bfmu^Y(\{t\} \times \dd y), \quad t \geq 0,
\end{equation}
satisfies
\begin{equation}\label{eq:hatUsupset}
J_Y \cap (\Omega \times [0,T]) \subset \{(\omega, t) \in \Omega \times (0,+\infty) \colon \hat U_t = 0\},
\end{equation}
 where $J_Y \coloneqq \{(\omega, t) \in \Omega \times (0,+\infty) \colon \hat \bfmu^Y(\omega; \, \{t\} \times \R) > 0\}$.
\end{theorem}

\begin{proof}
See Appendix \ref{Section_Appendix_7}.
\end{proof}


%
%

\section{The filtering equation}\label{sec:filter}

\subsection{\texorpdfstring{The filtering equation for $\Pi$}{The filtering equation for Pi}}\label{sec:filterPi}
In this section we derive the SDE that characterizes the filter $\Pi$, see Theorem \ref{th:filtering_equation}. 
We introduce the following random measures:
\begin{align}
\eta_f^i\bigl((0,t] \times B\bigr) &\coloneqq \int_0^t \biggl\{\int_H \Pi_{s^-}\bigl(f(\cdot) \, \ind_{d^{i,K}(B, s, \cdot, Y_{s^-})}(h) \,\bflambda^m(s, \cdot) \bfQ^m(s, \cdot ; \dd h)\bigr) \notag
\\
&\qquad + \int_Z \Pi_{s^-}\bigl(f(\cdot) \, \ind_{d^{i,G}(B, s, \cdot, Y_{s^-})}(z)  \,\bflambda^n(s, \cdot, Y_{s^-}) \bfQ^n(s, \cdot, Y_{s^-}; \dd z)\bigr) \biggr\} \, \dd s, \label{eq:etafi}
\\
\eta_f^{p,m}\bigl((0,t] \times B\bigr) &\coloneqq \int_0^t \int_H \Pi_{s^-}\bigl(f(\cdot) \, \ind_{d^{p,K}(B,s, \cdot, Y_{s^-})}(h) \,\bfR^m(s, \cdot; \dd h)\bigr) \, \dd p_s^m, \label{eq:etafpm}
\\
\eta_f^{p,n}\bigl((0,t] \times B\bigr) &\coloneqq \int_0^t \int_Z \Pi_{s^-}\bigl(f(\cdot) \, \ind_{d^{p,G}(B,s, \cdot, Y_{s^-})}(z) \,\bfR^n(s, \cdot, Y_{s^-}; \dd z)\bigr) \, \dd p_s^n, \label{eq:etafpn}
\\
\rho_f^i\bigl((0,t] \times B\bigr) &\coloneqq \int_0^t \int_H \Pi_{s^-}\bigl([f(h) - f(\cdot)] \ind_{d^{i,K}(B, s, \cdot, Y_{s^-})}(h) \,\bflambda^m(s, \cdot) \bfQ^m(s, \cdot ; \dd h)\bigr) \, \dd s, \label{eq:rhofi}
\\
\rho_f^{p,m}\bigl((0,t] \times B\bigr) &\coloneqq \int_0^t \int_H \Pi_{s^-}\bigl([f(h) - f(\cdot)] \, \ind_{d^{p,K}(B,s, \cdot, Y_{s^-})}(h) \,\bfR^m(s, \cdot; \dd h)\bigr) \, \dd p_s^m, \label{eq:rhofpm}
\end{align}
where $t > 0$, $B \in \cB(\R)$ and $f \colon H \rightarrow \R$ is any bounded and measurable function. We also introduce the  random measures:
\begin{align}
\eta^i\bigl((0,t] \times B\bigr) &\coloneqq \eta_\one^i\bigl((0,t] \times B\bigr), \label{eq:etai}
\\
\eta^{p,m}\bigl((0,t] \times B\bigr) &\coloneqq \eta_\one^{p,m}\bigl((0,t] \times B\bigr), \label{eq:etapm}
\\
\eta^{p,n}\bigl((0,t] \times B\bigr) &\coloneqq \eta_\one^{p,n}\bigl((0,t] \times B\bigr), \label{eq:etapn}
\end{align}
where $\one \colon H \to \R$ denotes the constant function equal to $1$ (cf. Section \ref{sec:notation}). We see that
$$
\hat \bfmu^Y(\dd t \, \dd y) = \eta^i(\dd t \, \dd y) + \eta^{p,m}(\dd t \, \dd y) + \eta^{p,n}(\dd t \, \dd y).
$$
Next, we define the following $\bbY$-predictable random sets:
\begin{align}
J_Y &\coloneqq \{(\omega,t) \in \Omega \times [0,+\infty) \colon \hat \bfmu^Y(\{t\} \times \R) > 0\}, \label{eq:JY}
\\
J_Y^m &\coloneqq \{(\omega,t) \in \Omega \times [0,+\infty) \colon \eta^{p,m}(\{t\} \times \R) > 0\}, \label{eq:JYm}
\\
J_Y^n &\coloneqq \{(\omega,t) \in \Omega \times [0,+\infty) \colon \eta^{p,n}(\{t\} \times \R) > 0\}. \label{eq:JYn}
\end{align}
Thanks to Assumption~\ref{hp:jointlaw}--\eqref{hp:m_n_disjoint}, the set $J_Y$ is the disjoint union of $J_Y^m$ and $J_Y^n$.

Finally, we define for all $t > 0$ operators $\cL_t \colon \dB_b(H) \to \dB_b(H)$ and $\cA_t \colon \dB_b(H) \to \dB_b(H)$ as
\begin{align}
\cL_t f(h') &\coloneqq \int_H \bigl[f(h) - f(h')\bigr] \, \bflambda^m(t, h') \, \bfQ^m(t, h'; \, \dd h), \quad h' \in H, \label{eq:opL} \\
\cA_t f(h') &\coloneqq \int_H \bigl[f(h) - f(h')\bigr] \, \bfR^m(t, h'; \, \dd h), \quad h' \in H. \label{eq:opA}
\end{align}

Our ultimate objective is to prove the following Theorem which provides the filtering equation for the process $\Pi$.

\begin{theorem}\label{th:filtering_equation}
Under Assumptions~\ref{hp:jointlaw}, \ref{hp:sdeexistuniq}, \ref{ass:girsanov}, and \ref{hp:martrepr}, for each $T > 0$ and for any bounded and measurable function $f \colon H \rightarrow \R$ the filter $\Pi$ satisfies the following SDE:
\begin{align}\label{eq:filter}
\Pi_t(f)&= \Pi_0(f) + \int_0^t \Pi_{s^-}\bigl(\cL_s f\bigr) \, \dd s + \int_0^t \Pi_{s^-}\bigl(\cA_s f\bigr) \, \dd p^m_s\\
&+ \int_0^t    \gamma_s\, \dd I_s+\int_0^t \int_\R U_s(y)  \, (m^Y- \hat\bfmu^Y)(\dd s\,\dd y), \quad t \in [0,T],
\end{align}
where
\begin{align}
&\gamma_t = \frac{\Pi_{t}\bigl( f(\cdot) \, \bfb(t, \cdot, Y_t)\bigr)- \Pi_t(f)\,\Pi_t\bigl(\bfb(t, \cdot, Y_t)\bigr)}{\sigma(t, Y_t)},  &\dd \bbP \otimes \dd t\text{-a.e.},\label{gamma}
\end{align}
and $U$ satisfies\footnote{Superscript $c$ in $J_Y^c$ stands for the complement set.}
\begin{equation}\label{eq:U}
U_t(y) = U^i_t(y) \, \ind_{J_Y^c}(t) + U^{p,m}_t(y) \, \ind_{J_Y^m}(t) + U^{p,n}_t(y) \, \ind_{J_Y^n}(t), \quad \dd \bbP \otimes \dd \hat \bfmu^Y\text{-a.e.},
\end{equation}
where
\begin{align}
U^i_t(y) &\coloneqq \frac{\dd \eta_f^i}{\dd \eta^i}(t,y)-\Pi_{t^-}(f)+ \frac{\dd \rho_f^i}{\dd \eta^i}(t,y), \label{eq:Ui}
\\
U^{p,m}_t(y) &\coloneqq \frac{\dd \eta_f^{p,m}}{\dd \eta^{p,m}}(t,y)-\Pi_{t^-}(f)+ \frac{\dd \rho_f^{p,m}}{\dd \eta^{p,m}}(t,y), \label{eq:Upm}
\\
U^{p,n}_t(y) &\coloneqq \frac{\dd \eta_f^{p,n}}{\dd \eta^{p,n}}(t,y)-\Pi_{t^-}(f). \label{eq:Upn}
\end{align}
\end{theorem}

\begin{remark}
Theorem \ref{th:filtering_equation} above provides existence of the filtering equation, but does not guarantee uniqueness of the solution. The latter is important in many applications, as for instance in control theory under partial information. However, due to the specific features of our model, uniqueness cannot be obtained by a direct application of existing results and for this reason we have therefore  decided to postpone the issue of uniqueness to a future work.
\end{remark}

The proof of Theorem \ref{th:filtering_equation} is based on the {\em innovations approach}. This technique allows to derive the dynamics of the filter as a solution of a stochastic differential equation driven by the the pair $(I, m^Y-\hat \bfmu^Y)$ (see Section~\ref{sec:martreprth}). The term {\em innovation} stands for the fact that the pair $(I,m^Y-\hat \bfmu^Y)$ carries any new available information on the signal.

We first state a preliminary lemma, 
whose proof is given in Appendix \ref{Section_Appendix_6}.

\begin{lemma}\label{L: martY}
Under Assumption~\ref{hp:jointlaw}, for any bounded and measurable $f \colon H \to \R$, the process $M^f \coloneqq (M_t^f)_{t \geq 0}$ defined by
	\begin{equation}\label{eq:Mtf}
		M_t^f \coloneqq \bbE \biggl[\int_0^t \int_H [f(h)- f(\bfX_{s-})]\,\bfmu(\dd s\, \dd h) \bigg |\cY_t\biggr] - \int_0^t \Pi_{s^-}\bigl(\cL_s f\bigr) \, \dd s - \int_0^t \Pi_{s^-}\bigl(\cA_s f\bigr) \, \dd p^m_s
	\end{equation}
is a $\bbY$-martingale.
\end{lemma}

The remainder of this section is dedicated to the proof of Theorem \ref{th:filtering_equation}.
In order to help readability and to underline the differences with the classical results we summarize here the content of the proof.

The structure of filtering equation~\eqref{eq:filter} is similar to existing results in the literature, with noise source given by the innovation $(I, m^Y - \hat \bfmu^Y)$ (see Section~\ref{sec:martreprth}). In particular, the integral with respect to the measure $m^Y - \hat \bfmu^Y$ can be decomposed into three integrals with respect to specific compensated random measures, accounting for totally inaccessible jumps of $Y$ and predictable jumps of $Y$, the latter given by $p^m$ and $p^n$. In each of these integrals, integrand processes \eqref{eq:Ui}, \eqref{eq:Upm} and \eqref{eq:Upn} have the typical structure given by the sum of three components, as in~\citet[Ch. IV, (1.15)]{bremaud:pp}.

Moving to more technical aspects, the proof consists of three steps:
\begin{itemize}
\item[(i)] In Step 1 we provide the semimartingale representation of the conditional expectation of the process the process $\bigl(f(\bfX_t)\bigr)_{t \geq 0}$ by projecting $f(\bfX_t)$ over the $\sigma$-algebra $\cY_t$. This part is standard: we use the Martingale Representation Theorem \ref{T:mtg_rep} to provide the general structure of the filtering equation in terms of processes $\gamma$ and $U$, as in equation~\eqref{eq:martrepr}, to be determined explicitly. 
\item[(ii)] In Step 2 we characterize the process $\gamma$. To do this, we use the identity
\begin{align}\label{eq:tower2}
\mathbb{E}\left[f(\bfX_t)  \widetilde W_t | \cY_t\right] = \Pi_t(f) \, \widetilde W_t,
\end{align}
for a suitable $\bbY$-adapted process $\widetilde W$. In this step, although the general procedure of the innovations approach stays the same, particular care must be taken to handle predictable jumps in the signal. In fact, the processes on the right and left hand sides of equality \eqref{eq:tower2} need to have the same special semimartingale decomposition and hence the same finite variation and local martingale parts. To identify local martingales, the hypothesis that $p^m$ is $\bbY$-predictable together with condition~\eqref{eq:ptmbdd} are fundamental.
\item[(iii)] In Step 3 we compute the process $U$. This step uses the identity
\begin{align}\label{eq:tower3}
\mathbb{E}\left[f(\bfX_t)  \psi_t | \cY_t\right] = \Pi_t(f) \, \psi_t,
\end{align}
with a suitable $\bbY$-adapted process $\psi$.
The most delicate and technical part of the proof is again to identify local martingales for the processes in equality~\eqref{eq:tower3}. Here we extensively use the crucial fact that both $p^m$ and $p^n$ are $\bbY$-predictable.
Finally, using that random sets $J_Y^m$, $J^n_Y$ and $J_Y^c$ in \eqref{eq:JYm}, \eqref{eq:JYn} and \eqref{eq:JY}, respectively, are pairwise disjoint allows us to decompose $U$ into three parts~\eqref{eq:Ui}, \eqref{eq:Upm} and \eqref{eq:Upn}.
\end{itemize}

\begin{proof}[Proof of Theorem \ref{th:filtering_equation}]
Fix $f \colon H \to \R$ bounded and measurable and $T > 0$.
By definition of the filter $\Pi$ given in~\eqref{eq:Pidef}, we need to compute the $\bbY$-optional projection of the process $\bigl(f(\bfX_t)\bigr)_{t \geq 0}$.

In order to ease the presentation of the proof we divide it into several steps.

\medskip
\noindent\textbf{Step 1. (Semimartingale representation of the filter).}

Consider first the $\bbF$-semimartingale decomposition of the process $\bigl(f(\bfX_t)\bigr)_{t \geq 0}$. Since this is a real-valued pure-jump process, whose jumps are governed by those of $\bfX$, we easily obtain
\begin{align}\label{dec}
f(\bfX_{t}) &= f(\bfX_{0})+ \int_0^t \int_H [f(h)- f(\bfX_{s-})]\,\bfmu(\dd s\, \dd h) \notag\\
&+ \int_0^t \int_H [f(h)- f(\bfX_{s-})]\,(\bfm-\bfmu)(\dd s\, \dd h).
\end{align}
Then the process
\begin{align*}
f(\bfX_{t}) - f(\bfX_{0})- \int_0^t \int_H [f(h)- f(\bfX_{s-})]\,\bfmu(\dd s\, \dd h), \qquad t \geq 0
\end{align*}
is an $\bbF$-martingale. Taking conditional expectation with respect to $\cY_t$ in \eqref{dec} we get
\begin{align*}
&\Pi_{t}(f) - \bbE[f(\bfX_{0})|\cY_t]- \bbE \biggl[\int_0^t \int_H [f(h)- f(\bfX_{s-})]\,\bfmu(\dd s\, \dd h) \bigg |\cY_t\biggr]\\
& = \bbE \biggl[\int_0^t \int_H [f(h)- f(\bfX_{s-})]\,(\bfm-\bfmu)(\dd s\, \dd h)\bigg |\cY_t\biggr], \quad t \geq 0.
\end{align*}
Next, define for any $t \geq 0$
\begin{align*}
\dI^{(1)}_t &\coloneqq \bbE[f(\bfX_{0}) \mid \cY_t] - \Pi_0(f),\\
\dI^{(2)}_t &\coloneqq \bbE \biggl[\int_0^t \int_H [f(h)- f(\bfX_{s-})]\,\bfmu(\dd s\, \dd h) \bigg |\cY_t\biggr],\\
\dI^{(3)}_t &\coloneqq \bbE \biggl[\int_0^t \int_H [f(h)- f(\bfX_{s-})]\,(\bfm-\bfmu)(\dd s\, \dd h)\bigg |\cY_t\biggr].
\end{align*}
The processes $\dI^{(1)}$ and $\dI^{(3)}$ are $\bbY$-martingales. For the process $\dI^{(2)}$, by applying  Lemma~\ref{L: martY} we get
\begin{equation*}
\dI^{(2)}_t = M_t^f + \int_0^t \Pi_{s^-}\bigl(\cL_s f\bigr) \, \dd s + \int_0^t \Pi_{s^-}\bigl(\cA_s f\bigr) \, \dd p^m_s, \quad t \geq 0,
\end{equation*}
where $M^f$ is the $\bbY$-martingale defined in~\eqref{eq:Mtf}.
By the martingale representation theorem (Theorem~\ref{T:mtg_rep}), there exist a $(\cY_t)_{t \in [0,T]}$-predictable random field $U$ and a $(\cY_t)_{t \in [0,T]}$-progressive process $\gamma$ satisfying \eqref{int_cond}, such that
\begin{align}\label{Piformula}
\Pi_t(f)-\Pi_0(f)- \int_0^t \Pi_{s^-}\bigl(\cL_s f\bigr) \, \dd s - \int_0^t \Pi_{s^-}\bigl(\cA_s f\bigr) \, \dd p^m_s
\notag\\
= \int_0^t \gamma_s \, \dd I_s + \int_0^t \int_\R U_s(y)\, (m^Y - \hat \bfmu^Y)(\dd s\, \dd y), \quad t \in [0,T].
\end{align}
The goal now is to identify processes $\gamma$ and $U$.

\newpage

\noindent\textbf{Step 2. (Identification of $\gamma$).}

Let $\widetilde W$ be the $\bbY$-adapted process defined as (cf.~\eqref{eq:Wtilde})
$$
\widetilde W_t : = W_t + \int_0^t \frac{\bfb(s, \bfX_s, Y_s)}{\sigma(s, Y_s)} \dd s=I_t+ \int_0^t \frac{\Pi_{s}(\bfb(s, \cdot, Y_s))}{\sigma(s, Y_s)} \dd s, \quad t \in [0,T].
$$

We need to compute in two distinct ways the $(\cY_t)_{t \in [0,T]}$-optional projection of the process $\bigl(f(\bfX_t) \, \widetilde W_t\bigr)_{t \in [0,T]}$. This provides us two distinct (special) semimartingale decompositions of the process $\bigl(\Pi_t(f) \, \widetilde W_t\bigr)_{t \in [0,T]}$ with respect to $(\cY_t)_{t \in [0,T]}$ and, thanks to uniqueness of this kind of decomposition, we get the expression for $\gamma$ comparing their finite variations parts.

On the one hand, using the integration by parts formula (see, e.g., \citep[Th. 14.1.1]{cohen:stochcalculus}) we get\footnote{If $M$ and $N$ are two semimartingales, the symbol $[M,N]$ denotes the (optional) quadratic covariation of $M$ and $N$.}
\begin{equation}
\begin{aligned}
f(\bfX_{t}) \, \widetilde W_t &=\int_0^t   \widetilde W_{s-}  \, \dd(f(\bfX_s)) +  \int_0^t f(\bfX_{s})\,  \dd \widetilde W_s + \bigl[\widetilde W, f(\bfX_\cdot)]_t\\
&= \int_0^t \widetilde W_{s-}  \int_H [f(h)- f(\bfX_{s-})]\,\bfmu(\dd s\, \dd h) \\
&+  \int_0^t \widetilde W_{s-}  \int_H [f(h)- f(\bfX_{s-})]\,(\bfm-\bfmu)(\dd s\, \dd h)\\
&+ \int_0^t f(\bfX_{s^-})\,  \dd  W_s  + \int_0^t f(\bfX_{s})\,  \frac{\bfb(s, \bfX_s, Y_s)}{\sigma(s, Y_s)} \, \dd s.
\end{aligned}
\end{equation}
Applying the conditional expectation with respect to $\cY_t$, the previous equality yields
\begin{align}
\bbE[f(\bfX_{t}) \, \widetilde W_t|\cY_t]
&= \bbE\biggr[\int_0^t  \widetilde W_{s-}  \int_H [f(h)- f(\bfX_{s-})]\,\bfmu(\dd s\, \dd h)\bigg|\cY_t\biggl] \notag
\\
&+  \bbE\biggr[\int_0^t  \widetilde W_{s-}  \int_H [f(h)- f(\bfX_{s-})]\,(\bfm-\bfmu)(\dd s\, \dd h)\bigg|\cY_t\biggl] \notag
\\
& + \bbE\biggr[\int_0^t f(\bfX_{s})\,  \dd  W_s\bigg|\cY_t\biggl]
+  \bbE\biggr[\int_0^t f(\bfX_{s})\,  \frac{\bfb(s, \bfX_s, Y_s)}{\sigma(s, Y_s)} \, \dd s\bigg|\cY_t\biggl]. \label{Tocompare_1}
\end{align}
Notice that, since $\widetilde W_t$ is $\cY_t$-measurable, the right hand side of equality \eqref{Tocompare_1} is nothing but $\Pi_t(f) \, \widetilde W_t$.

Boundedness of $f$ implies that $\int_0^t f(\bfX_{s^-})\,  \dd  W_s$, $t \in [0,T]$, is a $(\cF_t)_{t \in [0,T]}$-martingale, hence its optional projection, appearing in~\eqref{Tocompare_1}, is a $(\cY_t)_{t \in [0,T]}$-martingale.
Thanks to this assumption on $f$ and to Assumption~\ref{hp:jointlaw}, we can show that $\int_0^t \widetilde W_{s-} \int_H [f(h)- f(\bfX_{s-})]\,(\bfm-\bfmu)(\dd s\, \dd h)$, $t \in [0,T]$, is a $(\cF_t)_{t \in [0,T]}$-local martingale. As a matter of fact, defining the $(\cF_t)_{t \in [0,T]}$-stopping times $\tau_n \coloneqq \sup\{t \in [0,T] \colon \widetilde W_t \leq n\}$ ,we have:
\begin{align*}
&\mathrel{\phantom{\leq}} \bbE\biggl| \int_0^{T \land \tau_n} \widetilde W_{s-} \int_H [f(h)- f(\bfX_{s-})]\,\bfmu(\dd s\, \dd h) \biggr| \\
&\leq 2 \sup_{h \in H} |f(h)| \biggl\{\sup_{(t,h) \in [0,+\infty) \times H} \bflambda^m(t,h) \, \bbE\biggl[\int_0^{T \land \tau_n} |\widetilde W_s| \, \dd s\biggr] + \bbE\biggl[\int_0^{T \land \tau_n} |\widetilde W_{s^-}| \, \dd p^m_s\biggr] \biggr\} \\
&\leq 2 \sup_{h \in H} |f(h)| \biggl\{n T \sup_{(t,h) \in [0,+\infty) \times H} \bflambda^m(t,h) + n \, \bbE[p^m_T] \biggr\} < +\infty.
\end{align*}
From this, it is easy to show that its optional projection, appearing in~\eqref{Tocompare_1}, is a $(\cY_t)_{t \in [0,T]}$-local martingale.

Following the same lines (or simple modifications) of the proof of Lemma~\ref{L: martY} we have that the process
\begin{multline*}
\bbE \biggl[\int_0^t \widetilde W_{s-} \int_H [f(h)- f(\bfX_{s-})]\,\bfmu(\dd s\, \dd h) \bigg |\cY_t\biggr] - \int_0^t \widetilde W_s \, \Pi_s\bigl(\cL_s f\bigr) \, \dd s \\
- \int_0^t \widetilde W_{s-} \Pi_{s^-}\bigl(\cA_s f\bigr) \, \dd p^m_s, \quad t \in [0,T],
\end{multline*}
is a $(\cY_t)_{t \in [0,T]}$-local martingale and that the process
\begin{equation*}
\bbE\biggr[\int_0^t f(\bfX_{s})\,  \frac{\bfb(s, \bfX_s, Y_s)}{\sigma(s, Y_s)} \, \dd s \bigg |\cY_t\biggr] - \int_0^t \frac{\Pi_s\bigl(f(\cdot) \bfb(s, \cdot, Y_s)\bigr)}{\sigma(s, Y_s)}\, \dd s, \quad t \in [0,T]
\end{equation*}
is a $(\cY_t)_{t \in [0,T]}$-martingale. Therefore~\eqref{Tocompare_1} can be rewritten as:
\begin{multline}\label{Tocompare_1bis}
\Pi_t(f) \, \widetilde W_t=\bbE[f(\bfX_{t}) \, \widetilde W_t|\cY_t] = \int_0^t \widetilde W_s \, \Pi_s\bigl(\cL_s f\bigr) \, \dd s + \int_0^t \widetilde W_{s-} \Pi_{s^-}\bigl(\cA_s f\bigr) \, \dd p^m_s
\\ +  \int_0^t \frac{\Pi_s\bigl(f(\cdot) \bfb(s, \cdot, Y_s)\bigr)}{\sigma(s, Y_s)}\, \dd s + M^{(1)}_t,
\end{multline}
where $(M^{(1)}_t)_{t \in [0,T]}$ is a $(\cY_t)_{t \in [0,T]}$-local martingale.

On the other hand, since $\dd \widetilde W_t = \dd I_t + \int_0^t \frac{\Pi_s(\bfb(s, \cdot, Y_s))}{\sigma(s, Y_s)}   \,\dd s,$ (see equation \eqref{Innov}),
we have
\begin{equation*}
\begin{aligned}
\Pi_t(f) \, \widetilde W_t &= \int_0^t \Pi_{s^-}(f) \, \dd \widetilde W_s + \int_0^t \widetilde W_s \, \dd \Pi_s(f) + \bigl[\widetilde W, \Pi(f)]_t
\\
&= \int_0^t \Pi_{s}(f)\,  \dd I_s+  \int_0^t \Pi_{s}(f)\, \frac{\Pi_s(\bfb(s, \cdot, Y_s))}{\sigma(s, Y_s)}   \,\dd s
\\	
&+\int_0^t   \widetilde W_s  \, \Pi_s\bigl(\cL_s f\bigr)\,\dd s
+\int_0^t   \widetilde W_{s-}  \, \Pi_{s^-}\bigl(\cA_s f\bigr) \,\dd p^m_s
+  \int_0^t  \widetilde W_{s} \, \gamma_s \, \dd I_s + \int_0^t \gamma_s \, \dd s,
\end{aligned}
\end{equation*}
i.e., with a similar reasoning as before,
\begin{multline}\label{Tocompare_2}
\Pi_t(f) \, \widetilde W_t = \int_0^t \widetilde W_s \, \Pi_s\bigl(\cL_s f\bigr) \, \dd s + \int_0^t \widetilde W_{s-} \Pi_{s^-}\bigl(\cA_s f\bigr) \, \dd p^m_s \\
+ \int_0^t \Pi_{s}(f)\, \frac{\Pi_s(\bfb(s, \cdot, Y_s))}{\sigma(s, Y_s)} \,\dd s + \int_0^t \gamma_s \, \dd s + M^{(2)}_t, \quad t \in [0,T],
\end{multline}
where $(M^{(2)}_t)_{t \in [0,T]}$ is a $(\cY_t)_{t \in [0,T]}$-local martingale.

Formulas~\eqref{Tocompare_1bis} and~\eqref{Tocompare_2} entail that $(\Pi_t(f) \, \widetilde W_t)_{t \in [0,T]}$ is a $(\cY_t)_{t \in [0,T]}$-special semimartingale and, by uniqueness of the decomposition, the finite variation parts appearing in these two formulas must be equal. Hence
\begin{equation*}
\int_0^t \frac{\Pi_s\bigl(f(\cdot) \bfb(s, \cdot, Y_s)\bigr)}{\sigma(s, Y_s)}\, \dd s = \int_0^t \Pi_{s}(f)\, \frac{\Pi_s(\bfb(s, \cdot, Y_s))}{\sigma(s, Y_s)} \,\dd s + \int_0^t \gamma_s \, \dd s, \quad t \in [0,T],
\end{equation*}
whence the expression of $\gamma$ given in~\eqref{gamma}. Standard computations show that it satisfies~\eqref{int_cond}.

\medskip
\noindent\textbf{Step 3. (Identification of $U$)}.

Let us now turn to the process $U$. Consider any $\bbY$-adapted process $\psi = (\psi_t)_{t \geq 0}$ of the form $\psi_t \coloneqq \int_0^t \int_\R C_s (y)\, \bigl[m^Y-\hat\bfmu^Y](\dd s\, \dd y)$, where $C \colon \Omega \times [0,+\infty) \times \R \rightarrow \R$ is a $\bbY$-predictable bounded random field.
From~\citep[Th. 11.17]{he:semimart} and thanks to Assumptions~\ref{hp:jointlaw}-(ii)(iii) (see also point b. of Remark~\ref{rem:m_noexplosion}, Remark~\ref{rem:Dm} and~\eqref{eq:hatUsupset}), one can choose $U$ and $C$ so that the processes:
\begin{equation*}
\int_\R U_t(y) \, \hat \bfmu^Y(\{t\} \times \dd y) \, \text{ and } \, \int_\R C_t(y) \, \hat \bfmu^Y(\{t\} \times \dd y), \quad t \in [0,T],
\end{equation*}
are zero on the set $J_Y \cap (\Omega \times [0,T])$. This fact is used in what follows without further mention.

Similarly to what we did in the previous step, we need to compute in two distinct ways the $(\cY_t)_{t \in [0,T]}$-optional projection of the process $\bigl(f(\bfX_t) \, \psi_t\bigr)_{t \in [0,T]}$, then to compare the finite variation parts appearing in the semimartingale decompositions obtained.
On the one hand, using the integration by parts formula we get
{
\allowdisplaybreaks
\begin{align}
\nonumber f(\bfX_{t}) \, \psi_t
&= \int_0^t f(\bfX_{s^-}) \, \dd \psi_s + \int_0^t \psi_{s^-} \, \dd \bigl(f(\bfX_{s})\bigr) + \bigl[\psi, f(\bfX_\cdot)\bigr]_t
\\
\nonumber &= \int_0^t \int_\R f(\bfX_{s^-}) \, C_s(y) \, \bigl[m^Y-\bfmu^Y\bigr](\dd s\, \dd y)
\\
\nonumber&+ \int_0^t \int_\R f(\bfX_{s^-}) \, C_s(y) \, \bigl[\bfmu^Y-\hat\bfmu^Y\bigr](\dd s\, \dd y)
\\
\nonumber&+ \int_0^t \int_H \psi_{s^-} \, [f(h)- f(\bfX_{s-})]\,\bfmu(\dd s\, \dd h)
\\
&+ \int_0^t \int_H \psi_{s^-} \, [f(h)- f(\bfX_{s-})]\,(\bfm-\bfmu)(\dd s\, \dd h)
+\sum_{s \leq t} \Delta f(\bfX_s) \, \Delta \psi_s. \quad \label{eq:fPsiSmg}
\end{align}
}

Let us now analyze the last term appearing in this formula. It is clear that the processes $f(\bfX_\cdot)$ and $\psi$ jump simultaneously only at jump times relative to the coefficients $\bfK^i$ and $\bfK^p$ in the SDE of $Y$. Therefore, we have:
{
\allowdisplaybreaks
\begin{align*}
&\mathrel{\phantom{=}} \sum_{s \leq t} \Delta f(\bfX_s) \, \Delta \psi_s \\
&= \sum_{s \leq t} \biggl\{\int_H [f(h)- f(\bfX_{s-})]\,\bfm(\{s\} \times \dd h) \biggr\} \, \biggl\{\int_\R C_s (y)\, \bigl[m^Y-\hat\bfmu^Y](\{s\} \times \dd y)\biggr\} \\
&= \sum_{s \leq t} \biggl[ \ind_{D^m \setminus J^m}(s) \int_H \ind_{\bfK_s^i(h) \neq 0} \, [f(h)- f(\bfX_{s-})] \, C_s\bigl(\bfK_s^i(h)\bigr) \, \bfm(\{s\} \times \dd h) \\
&\qquad + \ind_{J^m}(s) \int_H \ind_{\bfK_s^p(h) \neq 0} \, [f(h)- f(\bfX_{s-})] \, C_s\bigl(\bfK_s^p(h)\bigr) \, \bfm(\{s\} \times \dd h) \biggl]\\
&= \int_0^t \int_H [f(h)- f(\bfX_{s-})] \, C_s\bigl(\bfK_s^i(h)\bigr) \, \ind_{\bfK_s^i(h) \neq 0} \, \bfm^i(\dd s \, \dd h) \\
&\qquad + \int_0^t \int_H [f(h)- f(\bfX_{s-})] \, C_s\bigl(\bfK_s^p(h)\bigr) \, \ind_{\bfK_s^p(h) \neq 0} \, \bfm^p(\dd s \, \dd h).
\end{align*}
}

Putting together this formula with~\eqref{eq:fPsiSmg} and taking the conditional expectation with respect to $\cY_t$ we get:
{
\allowdisplaybreaks
\begin{align}
&\mathrel{\phantom{=}} \Pi_t(f) \, \psi_t
= \bbE[f(\bfX_{t}) \, \psi_t \,|\, \cY_t] = \bbE \biggl[\int_0^t \int_\R f(\bfX_{s^-}) \, C_s(y) \, \bigl[m^Y-\bfmu^Y\bigr](\dd s\, \dd y) \biggm| \cY_t \biggr] \notag
\\
&+ \bbE \biggl[\int_0^t \int_\R f(\bfX_{s^-}) \, C_s(y) \, \bigl[\bfmu^Y-\hat\bfmu^Y\bigr](\dd s\, \dd y) \biggm| \cY_t \biggr] \notag
\\
&+ \bbE \biggl[\int_0^t \int_H \psi_{s^-} \, [f(h)- f(\bfX_{s-})]\,\bfmu(\dd s\, \dd h) \biggm| \cY_t \biggr] \notag
\\
&+ \bbE \biggl[\int_0^t \int_H \psi_{s^-} \, [f(h)- f(\bfX_{s-})]\,(\bfm-\bfmu)(\dd s\, \dd h) \biggm| \cY_t \biggr] \notag
\\
&+ \bbE \biggl[\int_0^t \int_H [f(h)- f(\bfX_{s-})] \, C_s\bigl(\bfK_s^i(h)\bigr) \, \ind_{\bfK_s^i(h) \neq 0} \, (\bfm^i - \bfmu^i)(\dd s \, \dd h) \biggm| \cY_t \biggr] \notag
\\
&+ \bbE \biggl[\int_0^t \int_H [f(h)- f(\bfX_{s-})] \, C_s\bigl(\bfK_s^p(h)\bigr) \, \ind_{\bfK_s^p(h) \neq 0} \, (\bfm^p - \bfmu^p)(\dd s \, \dd h) \biggm| \cY_t \biggr] \notag
\\
&+ \bbE \biggl[\int_0^t \int_H [f(h)- f(\bfX_{s-})] \, C_s\bigl(\bfK_s^i(h)\bigr) \, \ind_{\bfK_s^i(h) \neq 0} \, \bfmu^i(\dd s \, \dd h) \biggm| \cY_t \biggr] \notag
\\
&+ \bbE \biggl[\int_0^t \int_H [f(h)- f(\bfX_{s-})] \, C_s\bigl(\bfK_s^p(h)\bigr) \, \ind_{\bfK_s^p(h) \neq 0} \, \bfmu^p(\dd s \, \dd h) \biggm| \cY_t \biggr]
\label{eq:PiPsiSmg}
\end{align}
}
where, for any $B \in \cB(H), \, t \geq 0$, we have set
\begin{align*}
\bfmu^i\bigl((0,t] \times B\bigr) &\coloneqq \int_0^t \int_B \bflambda^m(s, \bfX_{s^-})\, \bfQ^m(s,\bfX_{s^-}; \dd h) \, \dd s,\\
\bfmu^p\bigl((0,t] \times B\bigr) &\coloneqq \int_0^t \int_B \bfR^m(s,\bfX_{s^-}; \dd h) \,\dd p^m_s.
\end{align*}

Thanks to the boundedness of process $C$, following the same lines of the proof of Lemma~\ref{L: martY} we have that the following four processes indicated by $(i)-(iv)$, defined for any $t \in [0,T]$, are $(\cY_t)_{t \in [0,T]}$-martingales
{
\allowdisplaybreaks
\begin{align*}
&(i) \hspace{.4cm} \bbE \biggl[\int_0^t \!\!\! \psi_{s^-} \! \! \!\int_H [f(h)- f(\bfX_{s-})]\,\bfmu(\dd s\, \dd h) \bigg |\cY_t\biggr] - \int_0^t\!\! \psi_{s^-} \Pi_s\bigl(\cL_s f\bigr) \, \dd s - \int_0^t \!\!\psi_{s^-} \Pi_{s^-}\bigl(\cA_s f\bigr) \, \dd p^m_s,
\\
&(ii) \hspace{.25cm} \bbE \biggl[\int_0^t \int_\R f(\bfX_{s^-}) \, C_s(y) \, \bigl[\bfmu^Y-\hat\bfmu^Y\bigr](\dd s\, \dd y) \biggm| \cY_t \biggr] \\
&\qquad  -\biggl[\int_0^t \biggl\{\int_H \Pi_{s^-}\bigl(f(\cdot) \, C_s\bigl(\bfK^i(s, \cdot, Y_{s^-}, h)\bigr) \, \ind_{d^{i,K}(s,\cdot,Y_{s^-})}(h) \,\bflambda^m(s, \cdot) \bfQ^m(s, \cdot ; \dd h)\bigr) \\
&\qquad \qquad+ \! \int_Z \! \Pi_{s^-}\bigl(f(\cdot) \, C_s\bigl(\bfG^i(s, \cdot, Y_{s^-}, h)\bigr) \, \ind_{d^{i,G}(s,\cdot,Y_{s^-})}(h) \,\bflambda^n(s, \cdot, Y_{s-}) \bfQ^n(s, \cdot, Y_{s-}; \dd z)\bigr) \biggr\} \, \dd s \\
&\qquad \quad + \int_0^t \int_H \Pi_{s^-}\bigl(f(\cdot) \, C_s\bigl(\bfK^p(s, \cdot, Y_{s^-}, h)\bigr) \, \ind_{d^{p,K}(s,\cdot,Y_{s^-})}(h) \,\bfR^m(s, \cdot; \dd h)\bigr) \, \dd p_s^m \\
&\qquad \quad + \int_0^t \int_Z \Pi_{s^-}\bigl(f(\cdot) \, C_s\bigl(\bfG^p(s, \cdot, Y_{s^-}, h)\bigr) \, \ind_{d^{p,G}(s,\cdot,Y_{s^-})}(h) \,\bfR^n(s, \cdot, Y_{s-}; \dd z)\bigr) \, \dd p_s^n \\
&\qquad \quad - \int_0^t \int_\R \Pi_{s^-}(f) \, C_s(y) \, \hat\bfmu^Y(\dd s\, \dd y) \biggr],
\\
&(iii) \hspace{.15cm} \bbE \biggl[\int_0^t \int_H [f(h)- f(\bfX_{s-})] \, C_s\bigl(\bfK_s^i(h)\bigr) \, \ind_{\bfK_s^i(h) \neq 0} \, \bfmu^i(\dd s \, \dd h) \biggm| \cY_t \biggr]
\\
&\quad -\int_0^t \int_H \Pi_{s^-} \bigl( [f(h)- f(\cdot)] \, C_s\bigl(\bfK^i(s, \cdot, Y_{s^-}, h)\bigr) \, \ind_{d^{i,K}(s,\cdot,Y_{s^-})}(h) \, \bflambda^m(s, \cdot)\, \bfQ^m(s,\cdot; \dd h) \bigr) \, \dd s,
\\
&(iv) \hspace{.2cm} \bbE \biggl[\int_0^t \int_H [f(h)- f(\bfX_{s-})] \, C_s\bigl(\bfK_s^p(h)\bigr) \, \ind_{\bfK_s^p(h) \neq 0} \, \bfmu^p(\dd s \, \dd h) \biggm| \cY_t \biggr]
\\
&\quad -\int_0^t \int_H \Pi_{s^-} \bigl( [f(h)- f(\cdot)] \, C_s\bigl(\bfK^p(s, \cdot, Y_{s^-}, h)\bigr) \, \ind_{d^{p,K}(s,\cdot,Y_{s^-})}(h) \, \bfR^m(s,\cdot; \dd h) \bigr) \, \dd p^m_s.
\end{align*}
}
Therefore~\eqref{eq:PiPsiSmg} can be rewritten as:
{
\allowdisplaybreaks
\begin{align}
&\mathrel{\phantom{=}} \Pi_t(f) \, \psi_t = \int_0^t \psi_{s^-} \, \Pi_s\bigl(\cL_s f\bigr) \, \dd s + \int_0^t \psi_{s^-} \Pi_{s^-}\bigl(\cA_s f\bigr) \, \dd p^m_s \notag
\\
&+ \int_0^t \biggl\{\int_H \Pi_{s^-}\bigl(f(\cdot) \, C_s\bigl(\bfK^i(s, \cdot, Y_{s^-}, h)\bigr) \, \ind_{d^{i,K}(s,\cdot,Y_{s^-})}(h) \,\bflambda^m(s, \cdot) \bfQ^m(s, \cdot ; \dd h)\bigr) \notag
\\
&\qquad+ \int_Z \Pi_{s^-}\bigl(f(\cdot) \, C_s\bigl(\bfG^i(s, \cdot, Y_{s^-}, h)\bigr) \, \ind_{d^{i,G}(s,\cdot,Y_{s^-})}(h) \,\bflambda^n(s, \cdot, Y_{s-}) \bfQ^n(s, \cdot, Y_{s-}; \dd z)\bigr) \biggr\} \, \dd s \notag
\\
&+ \int_0^t \int_H \Pi_{s^-}\bigl(f(\cdot) \, C_s\bigl(\bfK^p(s, \cdot, Y_{s^-}, h)\bigr) \, \ind_{d^{p,K}(s,\cdot,Y_{s^-})}(h) \,\bfR^m(s, \cdot; \dd h)\bigr) \, \dd p_s^m \notag
\\
&+ \int_0^t \int_Z \Pi_{s^-}\bigl(f(\cdot) \, C_s\bigl(\bfG^p(s, \cdot, Y_{s^-}, h)\bigr) \, \ind_{d^{p,G}(s,\cdot,Y_{s^-})}(h) \,\bfR^n(s, \cdot, Y_{s-}; \dd z)\bigr) \, \dd p_s^n \notag
\\
&+\int_0^t \int_H \Pi_{s^-} \bigl( [f(h)- f(\cdot)] \, C_s\bigl(\bfK^i(s, \cdot, Y_{s^-}, h)\bigr) \, \ind_{d^{i,K}(s,\cdot,y)}(h) \, \bflambda^m(s, \cdot)\, \bfQ^m(s,\cdot; \dd h) \bigr) \, \dd s \notag
\\
&+ \int_0^t \int_H \Pi_{s^-} \bigl([f(h)- f(\cdot)] \, C_s\bigl(\bfK^p(s, \cdot, Y_{s^-}, h)\bigr) \, \ind_{d^{p,K}(s,\cdot,y)}(h) \, \bfR^m(s,\cdot; \dd h) \bigr) \, \dd p^m_s \notag
\\
&- \int_0^t \int_\R \Pi_{s^-}(f) \, C_s(y) \, \hat\bfmu^Y(\dd s\, \dd y) + M^{(3)}_t,  \label{eq:PiPsiSmg2}
\end{align}
}for all $t \in [0,T]$, where $(M^{(3)}_t)_{t \in [0,T]}$ is a $(\cY_t)_{t \in [0,T]}$-martingale.

On the other hand, from~\eqref{Piformula}, we have that
\begin{align}
 \Pi_t(f) \, \psi_t
&= \int_0^t \Pi_{s^-}(f) \, \dd \psi_s + \int_0^t \psi_{s^-} \, \dd \Pi_s(f) + \bigl[\psi, \Pi(f)\bigr]_t
\\
&= \int_0^t \int_\R \Pi_{s^-}(f) \, C_s(y) \, \bigl[m^Y-\hat\bfmu^Y\bigr](\dd s\, \dd y)
\\
&+ \int_0^t \psi_{s^-} \, \Pi_{s^-}\bigl(\cL_s f\bigr) \, \dd s + \int_0^t \psi_{s^-} \, \Pi_{s^-}\bigl(\cA_s f\bigr) \, \dd p^m_s + \int_0^t \psi_s \, \gamma_s \, \dd I_s
\\
 & + \int_0^t \int_\R \psi_{s^-} \, U_s(y)\, (m^Y - \hat \bfmu^Y)(\dd s\, \dd y) + \sum_{s \leq t} \Delta \Pi_s(f) \, \Delta \psi_s, \label{eq:PiPsiSmg3}
\end{align}
for all $t \in [0,T]$.

With similar computations as above and using the properties of the processes $U$ and $C$, we can write:
\begin{equation*}
\sum_{s \leq t} \Delta \Pi_s(f) \, \Delta \psi_s = \int_0^t \int_\R U_s(y) \, C_s(y) \, m^Y(\dd s \, \dd y), \quad t \in [0,T],
\end{equation*}
hence~\eqref{eq:PiPsiSmg3} can be rewritten as
\begin{align}
 \Pi_t(f) \, \psi_t
&= \int_0^t \psi_{s^-} \, \Pi_{s^-}\bigl(\cL_s f\bigr) \, \dd s + \int_0^t \psi_{s^-} \, \Pi_{s^-}\bigl(\cA_s f\bigr) \, \dd p^m_s +
\\
& + \int_0^t \int_\R U_s(y) \, C_s(y) \, \hat \bfmu^Y(\dd s \, \dd y) + M^{(4)}_t, \label{eq:PiPsiSmg4}
\end{align}
for all $t \in [0,T]$, where $(M^{(4)}_t)_{t \in [0,T]}$ is a $(\cY_t)_{t \in [0,T]}$-martingale.

Formulas~\eqref{eq:PiPsiSmg2} and~\eqref{eq:PiPsiSmg4} entail that $(\Pi_t(f) \, \psi_t)_{t \in [0,T]}$ is a $(\cY_t)_{t \in [0,T]}$-special semimartingale and, by uniqueness of the decomposition, the finite variation parts appearing in these two formulas must be equal. Hence
{
\allowdisplaybreaks
\begin{align}
&\mathrel{\phantom{=}} \int_0^t \biggl\{\int_H \Pi_{s^-}\bigl(f(\cdot) \, C_s\bigl(\bfK^i(s, \cdot, Y_{s^-}, h)\bigr) \, \ind_{d^{i,K}(s,\cdot,Y_{s^-})}(h) \,\bflambda^m(s, \cdot) \bfQ^m(s, \cdot ; \dd h)\bigr) \notag
\\
&\qquad+ \int_Z \Pi_{s^-}\bigl(f(\cdot) \, C_s\bigl(\bfG^i(s, \cdot, Y_{s^-}, h)\bigr) \, \ind_{d^{i,G}(s,\cdot,Y_{s^-})}(h) \,\bflambda^n(s, \cdot, Y_{s-}) \bfQ^n(s, \cdot, Y_{s-}; \dd z)\bigr) \biggr\} \, \dd s \notag
\\
&+ \int_0^t \int_H \Pi_{s^-}\bigl(f(\cdot) \, C_s\bigl(\bfK^p(s, \cdot, Y_{s^-}, h)\bigr) \, \ind_{d^{p,K}(s,\cdot,Y_{s^-})}(h) \,\bfR^m(s, \cdot; \dd h)\bigr) \, \dd p_s^m
\notag\\
&+ \int_0^t \int_Z \Pi_{s^-}\bigl(f(\cdot) \, C_s\bigl(\bfG^p(s, \cdot, Y_{s^-}, h)\bigr) \, \ind_{d^{p,G}(s,\cdot,Y_{s^-})}(h) \,\bfR^n(s, \cdot, Y_{s-}; \dd z)\bigr) \, \dd p_s^n \notag
\\
&+\int_0^t \int_H \Pi_{s^-} \bigl([f(h)- f(\cdot)] \, C_s\bigl(\bfK^i(s, \cdot, Y_{s^-}, h)\bigr) \, \ind_{d^{i,K}(s,\cdot,y)}(h) \, \bflambda^m(s, \cdot)\, \bfQ^m(s,\cdot; \dd h) \bigr) \, \dd s \notag
\\
&+ \int_0^t \int_H \Pi_{s^-} \bigl([f(h)- f(\cdot)] \, C_s\bigl(\bfK^p(s, \cdot, Y_{s^-}, h)\bigr) \, \ind_{d^{p,K}(s,\cdot,y)}(h) \, \bfR^m(s,\cdot; \dd h) \bigr) \, \dd p^m_s \notag
\\
&- \int_0^t \int_\R \Pi_{s^-}(f) \, C_s(y) \, \hat\bfmu^Y(\dd s\, \dd y) = \int_0^t \int_\R U_s(y) \, C_s(y) \, \hat \bfmu^Y(\dd s \, \dd y), \quad t \in [0,T]. \label{eq:Uequality}
\end{align}
}
We recall that this formula must hold for any bounded $\bbY$-predictable random field $C$. Therefore, to complete the identification of the random field $U$ it is convenient to proceed \emph{via} a monotone class argument.

Consider $C$ of the form $C_t(y) = c_t \ind_B(y)$, $t \geq 0$, where $B \in \cB(\R)$ and $(c_t)_{t \geq 0}$ is a bounded $\bbY$-predictable process. With this choice~\eqref{eq:Uequality} becomes:
{
\allowdisplaybreaks
\begin{align}
&\mathrel{\phantom{=}} \int_0^t \int_\R c_s \ind_B(y) \, [\eta_f^i(\dd s \, \dd y) - \Pi_{s^-}(f) \, \eta^i(\dd s \, \dd y) + \rho_f^i(\dd s \, \dd y)] \notag
\\
&+ \int_0^t \int_\R c_s \ind_B(y) \,  [\eta_f^{p,m}(\dd s \, \dd y) - \Pi_{s^-}(f) \, \eta^{p,m}(\dd s \, \dd y) + \rho_f^{p,m}(\dd s \, \dd y)] \notag
\\
&+ \int_0^t \int_\R c_s \ind_B(y) \,  [\eta_f^{p,n}(\dd s \, \dd y) - \Pi_{s^-}(f) \, \eta^{p,n}(\dd s \, \dd y)] \notag
\\
&= \int_0^t \int_\R U_s(y) \, c_s \, \ind_B(y) \, \hat \bfmu^Y(\dd s \, \dd y), \quad t \in [0,T]. \label{eq:Uequality2}
\end{align}
}
The integral on the r.h.s. of~\eqref{eq:Uequality2} can be written as:
{
\allowdisplaybreaks
\begin{align*}
&\mathrel{\phantom{=}} \int_0^t \int_\R U_s(y) \, c_s \, \ind_B(y) \, \hat \bfmu^Y(\dd s \, \dd y) = \int_0^t \int_\R \ind_{J_Y^c}(s) \, U_s(y) \, c_s \, \ind_B(y) \, \hat \bfmu^Y(\dd s \, \dd y) \\
&\qquad + \int_0^t \int_\R \ind_{J_Y^m}(s) \, U_s(y) \, c_s \, \ind_B(y) \, \hat \bfmu^Y(\dd s \, \dd y) + \int_0^t \int_\R \ind_{J_Y^n}(s) \, U_s(y) \, c_s \, \ind_B(y) \, \hat \bfmu^Y(\dd s \, \dd y) \\
&= \int_0^t \int_\R \ind_{J_Y^c}(s) \, U_s(y) \, c_s \, \ind_B(y) \, \eta^i(\dd s \, \dd y) + \int_0^t \int_\R \ind_{J_Y^m}(s) \, U_s(y) \, c_s \, \ind_B(y) \, \eta^{p,m}(\dd s \, \dd y) \\
&\qquad + \int_0^t \int_\R \ind_{J_Y^n}(s) \, U_s(y) \, c_s \, \ind_B(y) \, \eta^{p,n}(\dd s \, \dd y).
\end{align*}
}
Similarly the l.h.s. of~\eqref{eq:Uequality2} can be rewritten as:
{
\allowdisplaybreaks
\begin{align*}
&\mathrel{\phantom{=}} \int_0^t \int_\R \ind_{J_Y^c}(s) \, c_s \ind_B(y) \, [\eta_f^i(\dd s \, \dd y) - \Pi_{s^-}(f) \, \eta^i(\dd s \, \dd y) + \rho_f^i(\dd s \, \dd y)] \\
&+ \int_0^t \int_\R \ind_{J_Y^m}(s) \, c_s \ind_B(y) \,  [\eta_f^{p,m}(\dd s \, \dd y) - \Pi_{s^-}(f) \, \eta^{p,m}(\dd s \, \dd y) + \rho_f^{p,m}(\dd s \, \dd y)] \\
&+ \int_0^t \int_\R \ind_{J_Y^n}(s) \, c_s \ind_B(y) \,  [\eta_f^{p,n}(\dd s \, \dd y) - \Pi_{s^-}(f) \, \eta^{p,n}(\dd s \, \dd y)].
\end{align*}
}
Therefore, to verify~\eqref{eq:Uequality2} it is equivalent to satisfy the following three equalities:
\begin{gather*}
\begin{split}
\int_0^t \int_\R \ind_{J_Y^c}(s) \, c_s \ind_B(y) \, [\eta_f^i(\dd s \, \dd y) - \Pi_{s^-}(f) \, \eta^i(\dd s \, \dd y) + \rho_f^i(\dd s \, \dd y)] \\
= \int_0^t \int_\R \ind_{J_Y^c}(s) \, U_s(y) \, c_s \, \ind_B(y) \, \eta^i(\dd s \, \dd y),
\end{split}
\\
\begin{split}
\int_0^t \int_\R \ind_{J_Y^m}(s) \, c_s \ind_B(y) \,  [\eta_f^{p,m}(\dd s \, \dd y) - \Pi_{s^-}(f) \, \eta^{p,m}(\dd s \, \dd y) + \rho_f^{p,m}(\dd s \, \dd y)] \\
= \int_0^t \int_\R \ind_{J_Y^m}(s) \, U_s(y) \, c_s \, \ind_B(y) \, \eta^{p,m}(\dd s \, \dd y),
\end{split}
\\
\begin{split}
\int_0^t \int_\R \ind_{J_Y^n}(s) \, c_s \ind_B(y) \,  [\eta_f^{p,n}(\dd s \, \dd y) - \Pi_{s^-}(f) \, \eta^{p,n}(\dd s \, \dd y)] \\
= \int_0^t \int_\R \ind_{J_Y^n}(s) \, U_s(y) \, c_s \, \ind_B(y) \, \eta^{p,n}(\dd s \, \dd y).
\end{split}
\end{gather*}
If we decompose $U$ as in~\eqref{eq:U}, these three identities are verified. Hence, by a monotone class argument, we get that~\eqref{eq:Uequality} is satisfied for any bounded $\bbY$-predictable random field $C$ with this decomposition of $U$. Finally, it can be verified \emph{via} standard computations that $U$ satisfies~\eqref{int_cond} and this concludes our proof.
\end{proof}

\subsection{\texorpdfstring{The filtering equation for $\pi$}{The filtering equation for pi}}\label{sec:filterpi}
We conclude the section by writing the explicit SDE for the filter $\pi$, defined in \eqref{eq:pidef}, which provides in particular the conditional law of the signal process $X$ 
with respect to the observation $\sigma$-algebra $\cY_t$. 
This is obtained using the filtering equation for the filter $\Pi$, defined in \eqref{eq:Pidef} and the results of Section \ref{sec:markov}.

Before giving the equation for $\pi$, we need some definitions. We define the random measures:
\begin{align}
\eta_\phi^i\bigl((0,t] \times B\bigr) &\coloneqq \int_0^t \biggl\{\int_E \pi_{s^-}\bigl(\phi(\cdot) \, \ind_{B\setminus\{0\}}\left(K^i(s, \cdot, Y_{s^-},e)\right) \,\lambda^m(s, \cdot) Q^m(s, \cdot ; \dd e)\bigr) \notag
\\
&\qquad + \int_Z \pi_{s^-}\bigl(\phi(\cdot) \, \ind_{B\setminus\{0\}}\left(G^i(s, \cdot, Y_{s^-},z)\right)  \,\lambda^n(s, \cdot, Y_{s^-}) Q^n(s, \cdot, Y_{s^-}; \dd z)\bigr) \biggr\} \, \dd s, \label{eq:etaphii}
\\
\eta_\phi^{p,m}\bigl((0,t] \times B\bigr) &\coloneqq \int_0^t \int_E \pi_{s^-}\bigl(\phi(\cdot) \, \ind_{B\setminus\{0\}}\left(K^p(s, \cdot, Y_{s^-},e)\right) \,R^m(s, \cdot; \dd e)\bigr) \, \dd p_s^m, \label{eq:etaphipm}
\\
\eta_\phi^{p,n}\bigl((0,t] \times B\bigr) &\coloneqq \int_0^t \int_Z \pi_{s^-}\bigl(\phi(\cdot) \, \ind_{B\setminus\{0\}}\left(G^p(s, \cdot, Y_{s^-},z)\right) \,R^n(s, \cdot, Y_{s^-}; \dd z)\bigr) \, \dd p_s^n, \label{eq:etaphipn}
\\
\rho_\phi^i\bigl((0,t] \times B\bigr) &\coloneqq \int_0^t \int_E \pi_{s^-}\bigl([\phi(J_{\widetilde \cD_E}(\cdot,s,e)) - \phi(\cdot)] \ind_{B\setminus\{0\}}\left(K^i(s, \cdot, Y_{s^-},e)\right) \times \notag
\\
&\qquad \qquad \qquad \times \lambda^m(s, \cdot) Q^m(s, \cdot ; \dd e)\bigr) \, \dd s, \label{eq:rhophii}
\\
\rho_\phi^{p,m}\bigl((0,t] \times B\bigr) &\coloneqq \int_0^t \int_E \pi_{s^-}\bigl([\phi(J_{\widetilde \cD_E}(\cdot,s,e)) - \phi(\cdot)] \, \ind_{B\setminus\{0\}}\left(K^p(s, \cdot, Y_{s^-},e)\right)R^m(s, \cdot; \dd e)\bigr) \, \dd p_s^m, \label{eq:rhophipm}
\end{align}
where $t > 0$, $B \in \cB(\R)$ and $\phi \colon \cD_E \rightarrow \R$ is any bounded and measurable function. Notice that applying relations \eqref{eq:JHJDE}, \eqref{eq:JDEJH}, and definitions \eqref{D:bfQ}, \eqref{D:bfR}, \eqref{D:bflambda}, \eqref{eq:obscoeffs}, one can easily verify that $\eta_\phi^i = \eta_f^i$, $\eta_\phi^{p,m} = \eta_f^{p,m}$, $\eta_\phi^{p,n} = \eta_f^{p,n}$, $\rho_\phi^i = \rho_f^i$, $\rho_\phi^{p,m} = \rho_f^{p,m}$, $\bbP$-a.s., whenever $\phi = f \circ S^{-1}$ or, equivalently, $f = \phi \circ S$, where $S$ is the map given in Proposition \ref{P_S}.
Applying the same formulas, we get that the random measure
\begin{align}
\hat\mu^Y\bigl((0,t] \times B\bigr)
&= \int_0^t \biggl\{\int_E \pi_{s^-}\bigl(\ind_{B\setminus\{0\}}\left(K^i(s, \cdot, Y_{s^-},e)\right) \,\lambda^m(s, \cdot) Q^m(s, \cdot ; \dd e)\bigr)\bigr) \\
&\qquad + \int_Z \pi_{s^-}\bigl(\ind_{B\setminus\{0\}}\left(G^i(s, \cdot, Y_{s^-},z)\right)  \,\lambda^n(s, \cdot, Y_{s^-}) Q^n(s, \cdot, Y_{s^-}; \dd z)\bigr)\bigr) \biggr\} \, \dd s
\\
&+ \int_0^t \int_E \pi_{s^-}\bigl(\ind_{B\setminus\{0\}}\left(K^p(s, \cdot, Y_{s^-},e)\right) \,R^m(s, \cdot; \dd e)\bigr) \, \dd p_s^m \\
&+ \int_0^t \int_Z \pi_{s^-}\bigl(\ind_{B\setminus\{0\}}\left(G^p(s, \cdot, Y_{s^-},z)\right) \,R^n(s, \cdot, Y_{s^-}; \dd z)\bigr) \, \dd p_s^n, \label{eq:mYYcompnew}
\end{align}
coincides with random measure $\hat \bfmu^Y$, given in \eqref{eq:mYYcomp} and  therefore, it provides a version of the $\bbY$-dual predictable projection of random measure $m^Y$.

Finally, we define for all $t > 0$ operators $\dL_t \colon \dB_b(\cD_E) \to \dB_b(\widetilde \cD_E)$ and $\dA_t \colon \dB_b(\cD_E) \to \dB_b(\widetilde \cD_E)$ as
\begin{align}
\dL_t \phi(x) &\coloneqq \int_E \left[\phi(J_{\widetilde \cD_E}(x,t,e)) - \phi(x)\right] \, \lambda^m(t, x) \, Q^m(t, x; \, \dd e), \quad x \in \widetilde \cD_E, \label{eq:opLnew} \\
\dA_t \phi(x) &\coloneqq \int_E \left[\phi(J_{\widetilde \cD_E}(x,t,e)) - \phi(x)\right] \, R^m(t, x; \, \dd e), \quad x \in \widetilde \cD_E. \label{eq:opAnew}
\end{align}
By a similar reasoning, applying relations \eqref{eq:JHJDE}, \eqref{eq:JDEJH}, and definitions \eqref{D:bfQ}, \eqref{D:bfR}, \eqref{D:bflambda}, it holds that $\dL_t \phi = \cL_t f$ and $\dA_t \phi = \cA_t f$, whenever $\phi = f \circ S^{-1}$ or, equivalently, $f = \phi \circ S$.

Therefore, recalling also the equations \eqref{eq:filterequality} and \eqref{eq:JY}--\eqref{eq:JYn}, we get the following  corollary of Theorem \ref{th:filtering_equation}.
\begin{corollary}\label{cor:filterpi}
Under Assumptions \ref{hp:jointlaw}, \ref{hp:sdeexistuniq}, \ref{ass:girsanov}, and \ref{hp:martrepr}, for each $T > 0$ and for any bounded and measurable function $\phi \colon \cD_E \rightarrow \R$ the filter $\pi$, defined in \eqref{eq:pidef}, satisfies the SDE:
\begin{align}\label{eq:filterpi}
\pi_t(\phi)&= \pi_0(\phi) + \int_0^t \pi_{s^-}\bigl(\dL_s \phi\bigr) \, \dd s + \int_0^t \pi_{s^-}\bigl(\dA_s \phi\bigr) \, \dd p^m_s\\
&+ \int_0^t    \gamma_s\, \dd I_s+\int_0^t \int_\R U_s(y)  \, (m^Y- \hat \mu^Y)(\dd s\,\dd y), \quad t \in [0,T],
\end{align}
where
\begin{align}
&\gamma_t = \frac{\pi_{t}\bigl( \phi(\cdot) \, b(t, \cdot, Y_t)\bigr)- \pi_t(\phi)\,\pi_t\bigl(b(t, \cdot, Y_t)\bigr)}{\sigma(t, Y_t)},  &\dd \bbP \otimes \dd t\text{-a.e.},\label{gammanew}
\end{align}
and $U$ satisfies
\begin{equation}\label{eq:Unew}
U_t(y) = U^i_t(y) \, \ind_{J_Y^c}(t) + U^{p,m}_t(y) \, \ind_{J_Y^m}(t) + U^{p,n}_t(y) \, \ind_{J_Y^n}(t), \quad \dd \bbP \otimes \dd \hat \mu^Y\text{-a.e.},
\end{equation}
where
\begin{align}
U^i_t(y) &\coloneqq \frac{\dd \eta_\phi^i}{\dd \eta^i}(t,y)-\pi_{t^-}(\phi)+ \frac{\dd \rho_\phi^i}{\dd \eta^i}(t,y), \label{eq:Uinew}
\\
U^{p,m}_t(y) &\coloneqq \frac{\dd \eta_\phi^{p,m}}{\dd \eta^{p,m}}(t,y)-\pi_{t^-}(\phi)+ \frac{\dd \rho_\phi^{p,m}}{\dd \eta^{p,m}}(t,y), \label{eq:Upmnew}
\\
U^{p,n}_t(y) &\coloneqq \frac{\dd \eta_\phi^{p,n}}{\dd \eta^{p,n}}(t,y)-\pi_{t^-}(\phi). \label{eq:Upnnew}
\end{align}
\end{corollary}%
\color{blue}
\begin{remark}\label{rem:filternodiff}
	The same techniques can be replicated to cover the case where the dynamics of the observation process is given by 
	\begin{equation}\label{eq:Yode}
\left\{
\begin{aligned}
	\dd Y_t &= b_t \, \dd t+ \int_E K^i_t(e) \, m^i(\dd t \, \dd e) + \int_E K^p_t(e) \, m^p(\dd t \, \dd e) \\
	& + \int_Z G^i_t(z) \, n^i(\dd t \, \dd z) + \int_Z G^p_t( z) \, n^p(\dd t \, \dd z), \\
	Y_0 &= y_0 \in \R.
\end{aligned}
\right.
\end{equation}
In this case the innovation process is only given by the compensated measure  $m^Y- \hat \bfmu^Y$, and therefore Assumption \ref{ass:girsanov} is not needed. Moreover, Assumption \ref{hp:martrepr} is automatically satisfied, 
see e.g. \cite{jacod:mpp}. The filtering equation in this context reduces to 
\begin{align}\label{eq:filterpi2}
\quad \qquad\qquad \qquad\pi_t(\phi)&= \pi_0(\phi) + \int_0^t \pi_{s^-}\bigl(\dL_s \phi\bigr) \, \dd s + \int_0^t \pi_{s^-}\bigl(\dA_s \phi\bigr) \, \dd p^m_s\\
&+\int_0^t \int_\R U_s(y)  \, (m^Y- \hat \mu^Y)(\dd s\,\dd y), \quad t \in [0,T]. \quad \qquad \qquad \qquad\qed
\end{align}
\end{remark}
\normalcolor

\color{blue}
\section{Examples}\label{sec:examples}
In this section we present three examples which are covered by our general setting. For these examples we verify, first, that Assumptions \ref{hp:jointlaw}, \ref{hp:sdeexistuniq}, \ref{ass:girsanov}, and \ref{hp:martrepr} are satisfied and then we  write the filtering equations. 
\normalcolor
\begin{example}[Deterministic jump times of the signal]\label{example:det_times}
On a filtered probability space $(\Omega, \cF, \bbF, \bbP)$, we consider a \mbox{pure-jump} signal process $X$ taking values in a discrete and finite space $E$, whose jump times are deterministic, i.e. 
\begin{equation*}
X_t = \sum_{n \in \N_0} \zeta_n \ind_{[t_n, t_{n+1})}, \quad t \geq 0,
\end{equation*}
where
{\color{blue} $t_0 = 0$, $\zeta_0 = e \in E$, $(t_n)_{n \in \N} \subset (0,+\infty]$ is a deterministic sequence of time points, and $(\zeta_n)_{n \in \N}$ is a sequence of $E$-valued random variables satisfying, for all $n \in \N$, the conditions
\begin{itemize}
\item $t_n \leq t_{n+1}$ and $t_n < +\infty \, \Longrightarrow \, t_n < t_{n+1}$;
\item $\zeta_n$ is $\cF_{t_n}$-measurable;
\item $\lim_{n \to \infty} t_n = +\infty$.
\end{itemize}
We assume that $(\zeta_n)_{n \in \N}$ is an $E$-valued \mbox{discrete-time} Markov chain, with probability transition matrix $(r_{ij})_{i, \, j \in E}$, and for simplicity $r_{ii} = 0$, $i \in E$. The observed process $Y$ satisfies the SDE
\begin{equation*}
\left\{
\begin{aligned}
\dd Y_t &= b(t, X_{t \land \cdot}, Y_t) \, \dd t + 
\, \dd W_t, \quad t \geq 0, \\
Y_0 &= y \in \R,
\end{aligned}
\right.
\end{equation*}
where $b$ is a Lipschitz continuous and bounded function and $W$ is a standard Brownian motion independent of random variables $(\zeta_n)_{n \in \N}$.
This model satisfies Assumptions~\ref{hp:jointlaw}-\eqref{hp:Yunique}\eqref{hp:m_n_disjoint}, \ref{hp:sdeexistuniq}, \ref{ass:girsanov}, and \ref{hp:martrepr}.
}

We define the random counting measure $p_t^m \coloneqq \sum_{n \in \N} \ind_{t_n \leq t}$, $t > 0$. It is evident that $p^m$ is a $\bbY$-predictable random measure  and that, by the previous assumptions, $\bbE[p_t^m] = p_t^m < +\infty$, for all $t > 0$.

{\color{blue} The random counting measure $m$ associated to the signal process $X$ admits the $\bbF$-compensator $\mu$ given by
\begin{equation*}
\mu((0,t] \times \{j\}) = \int_0^t r_{X_{s^-} \, j} \, \dd p_s^m, \quad t > 0, \, j \in E,
\end{equation*}
Therefore, $R^m(t,x; \{j\}) = r_{x(t^-) \, j}$, for all $t > 0$, $x \in \cD_E$, and $j \in E$, and hence Assumption~\ref{hp:jointlaw}-\eqref{hp:predproj} is also verified.
}

$X$ is a pure jump process with deterministic jump times, whose positions are not observed, and  modify the drift coefficient of the above SDE. In this model there is knowledge of the exact times at which these changes happen, but the new position of $X$ is not known and needs to be estimated.

{\color{blue}
By applying Corollary~\ref{cor:filterpi} we get that for any $T > 0$ the filtering equation for $\pi$ is given by 
\begin{align*}
\pi_t(\phi) = \pi_0(\phi) + \int_0^t \pi_{s^-}\bigl(\dA_s \phi\bigr) \, \dd p^m_s + \int_0^t  \gamma_s\, \dd I_s, \quad t \in [0,T],
\end{align*}
for any bounded and measurable function $\phi \colon \cD_E \rightarrow \R$ 
where
\begin{equation*}
\dA_t \phi(x) \coloneqq \sum_{j \in E} \left[\phi(J_{\widetilde \cD_E}(x,t,j)) - \phi(x)\right] \, r_{x(t^-) \, j}, \quad x \in \widetilde \cD_E, \, t > 0,
\end{equation*}
and
\begin{align}
\gamma_t = 
\pi_{t}\bigl( \phi(\cdot) \, b(t, \cdot, Y_t)\bigr)- \pi_t(\phi)\,\pi_t\bigl(b(t, \cdot, Y_t)\bigr)
, \quad t \in [0,T].
\end{align}
}
\end{example}

\begin{example}[Predictable jumps of the signal triggered by a diffusion process]\label{ex:pred_jumps}
In this example the signal process $X$ is a pure-jump process with {\color{blue} a single} predictable (non-deterministic) jump time,
triggered by {\color{blue} the observed} process $Y$ {\color{blue} once it} reaches a threshold $\ell \in \mathbb{R}$.
{\color{blue} This jump time is therefore observable, but we do not observe post-jump position of $X$, and hence the pair $(X,Y)$ constitutes a partially observed system.}

{\color{blue}
To construct this pair of processes we consider a filtered probability space $(\Omega, \cF, \bbF, \bbP)$, let $E$ be a discrete and finite space  and  $\zeta_0, \zeta_1$  be a pair of $E$-valued random variables such that   
\begin{equation*}
\bbP(\zeta_1 = j \mid \zeta_0 = i) = r_{ij}, \quad i,j \in E,
\end{equation*}
where $(r_{ij})_{i, \, j \in E}$ is a probability transition matrix. We also assume for simplicity that $r_{ii} = 0$, $i \in E$.

Let $\widetilde b \colon [0,+\infty) \times E \times \R \to \R$
 be a Lipschitz continuous and bounded function, 
and consider, for each $s \geq 0$, $z \in \R$, and $j \in E$, the following SDE 
\begin{equation*}
\left\{
\begin{aligned}
\dd Z_t^{s,z,j} &= \widetilde b(t,j,Z_t^{s,z,j}) \, \dd t +\dd W_t, \quad t \geq s, \\
Z_s^{s,z,j} &= z \in \R,
\end{aligned}
\right.
\end{equation*}
where $W \coloneqq (W_t)_{t \geq 0}$ is a standard real-valued Brownian motion, independent of $\xi_0, \xi_1$.  
Let $y, \, \ell \in \R$, $\ell \neq y$, be fixed constants and define the stopping time
\begin{equation*}
\tau \coloneqq \inf\{t > 0 \colon Z_t^{0,y,\zeta_0} = \ell\}.
\end{equation*}
The signal and the observed processes, for all $t \geq 0$, are given by 
\begin{align*}
X_t &= \zeta_0 \ind_{t < \tau} + \zeta_1 \ind_{t \geq \tau}, \\
Y_t &= Z_t^{0,y,\zeta_0} \ind_{t < \tau} + Z_t^{\tau, \ell, \zeta_1} \ind_{t \geq \tau}.
\end{align*}

It is clear that, on $\{t < \tau\}$,
\begin{equation*}
Y_t = y + \int_0^t \widetilde b(s,\zeta_0,Y_s) \, \dd s +  W_t
= y + \int_0^t \widetilde b(s,X_s,Y_s) \, \dd s + W_t.
\end{equation*}
Instead, observing that $Y_\tau = Z_\tau^{\tau, \ell, \zeta_1} = \ell$, we have that, on $\{t \geq \tau\}$,
\begin{align*}
Y_t &= \ell + \int_\tau^t \widetilde b(s,\zeta_1,Y_s) \, \dd s + W_t-W_\tau \\
&= y + \int_0^\tau \widetilde b(s,\zeta_0,Y_s) \, \dd s + W_\tau + \int_\tau^t \widetilde b(s,\zeta_1,Y_s) \, \dd s + W_t - W_\tau
\\
&= y + \int_0^t \widetilde b(s,X_s,Y_s) \, \dd s + W_t.
\end{align*}
Therefore, defining for all $t > 0$, $x \in \cD_E$, $y \in \R$, the function $b(t,x,y) \coloneqq \widetilde b(t,x(t),y)$, the observed process $Y$ satisfies the SDE
\begin{equation*}
\left\{
\begin{aligned}
\dd Y_t &= b(t, X_{t \land \cdot}, Y_t) \, \dd t + \dd W_t, \quad t \geq 0, \\
Y_0 &= y \in \R.
\end{aligned}
\right.
\end{equation*}
The hypotheses  on the coefficient $\widetilde b$  ensure that Assumptions~\ref{hp:jointlaw}-\eqref{hp:Yunique}\eqref{hp:m_n_disjoint}, \ref{hp:sdeexistuniq}, \ref{ass:girsanov}, and \ref{hp:martrepr} are verified. 

As in Example \ref{example:det_times},  we define the random counting measure $m$ associated to the signal process $X$ with the $\bbF$-compensator
\begin{equation*}
\mu((0,t] \times \{j\}) = \int_0^t r_{X_{s^-} \, j} \, \dd p_s^m, \quad t > 0, \, j \in E,
\end{equation*}
where $p_t^m \coloneqq \ind_{\tau \leq t}$, $t \geq 0$, is $\bbY$-predictable.
We have that $R^m(t,x; \{j\}) = r_{x(t^-) \, j}$, for all $t > 0$, $x \in \cD_E$, and $j \in E$, and hence,  Assumption~\ref{hp:jointlaw}-\eqref{hp:predproj} holds.

Next, since the pair $(X,Y)$ is a partially observed system where the post-jump position of $X$, i.e., $\zeta_1$, is not observed, we need the filtering equation  to estimate it. 
Applying Corollary~\ref{cor:filterpi}, we get that, for any $T > 0$ and for any bounded and measurable function $\phi \colon \cD_E \rightarrow \R$ the filter $\pi$ satisfies the SDE
\begin{align*}
\pi_t(\phi) = \pi_0(\phi) + \int_0^t  \gamma_s\, \dd I_s + \ind_{t \geq \tau} \pi_{\tau^-}\bigl(\dA_\tau \phi\bigr), \quad t \in [0,T],
\end{align*}
where
\begin{equation*}
\dA_t \phi(x) \coloneqq \sum_{j \in E} \left[\phi(J_{\widetilde \cD_E}(x,t,j)) - \phi(x)\right] \, r_{x(t^-) \, j}, \quad x \in \widetilde \cD_E, \, t > 0,
\end{equation*}
and
\begin{align}
&\gamma_t = \pi_{t}\bigl( \phi(\cdot) \, b(t, \cdot, Y_t)\bigr)- \pi_t(\phi)\,\pi_t\bigl(b(t, \cdot, Y_t)\bigr),  \quad t \in [0,T].
\end{align}
}
\end{example}

{\color{blue}
\begin{example}[Predictable jumps of the signal and the observation]\label{ex:ODEjumps}
In this last example, we follow a construction of the pair $(X,Y)$ analogous to that of Example~\ref{ex:pred_jumps}. The signal process $X$ is still a pure-jump process with a single predictable jump time,
triggered by the observed process $Y$ as it reaches a threshold $\ell \in \mathbb{R}$. In contrast  to the previous example, $Y$ is a non-diffusive process satisfying an equation of the form~\eqref{eq:Yode}. 

We consider a filtered probability space $(\Omega, \cF, \bbF, \bbP)$. Let $E$ be a discrete and finite space and $\zeta_0, \zeta_1$ be a pair of $E$-valued random variables such that  
\begin{equation*}
\bbP(\zeta_1 = j \mid \zeta_0 = i) = r_{ij}, \quad i,j \in E,
\end{equation*}
where $(r_{ij})_{i, \, j \in E}$ is a probability transition matrix and, for simplicity, $r_{ii} = 0$, $i \in E$.

Let $\widetilde b \colon [0,+\infty) \times E \times \R \to \R$ be a Lipschitz continuous and bounded function, 
and consider, for each $s \geq 0$, $z \in \R$, and $j \in E$, the following ODE 
\begin{equation*}
\left\{
\begin{aligned}
\dd Z_t^{s,z,j} &= \widetilde b(t,j,Z_t^{s,z,j}) \, \dd t, \quad t \geq s, \\
Z_s^{s,z,j} &= z \in \R.
\end{aligned}
\right.
\end{equation*}
We also consider two bounded measurable function $K^p \colon (0,+\infty) \times E \times E \to \R \setminus \{0\}$, $G^p \colon (0,+\infty) \times E \times \R \times \R \to \R \setminus \{0\}$, and  a  probability transition kernel $\widetilde R^n \colon (0,+\infty) \times E \times \R \to \cP(\R)$.
Let $y, \, \ell \in \R$, $\ell \neq y$, be fixed constants and set 
\begin{align*}
\tau_1 &\coloneqq \inf\{t > 0 \colon Z_t^{0,y,\zeta_0} = \ell\},\\
\eta_1 &\coloneqq K^p(\tau_1, \zeta_0, \zeta_1),\\
\tau_2 &\coloneqq \inf\{t > \tau_1 \colon Z_t^{\tau_1,\ell+\eta_1,\zeta_1} = \ell\},\\
\eta_2 &\coloneqq G^p(\tau_2, \zeta_1, Z_{\tau_2^-}^{\tau_1, \ell + \eta_1, \zeta_1}, \theta_2),
\end{align*}
where $\theta_2$ is another real random variable 
such that
\begin{equation*}
\bbP(\theta_2 \in B \mid \tau_1, \zeta_1, Z_{\tau_2^-}^{\tau_1, \ell + \eta_1, \zeta_1}) = \widetilde R^n(\tau_2, \zeta_1, Z_{\tau_2^-}^{\tau_1, \ell + \eta_1, \zeta_1}; B), \quad B \in \cB(\R).
\end{equation*}

Next we define the signal and the observed processes, for all $t \geq 0$, as
\begin{align*}
X_t &= \zeta_0 \ind_{t < \tau_1} + \zeta_1 \ind_{t \geq \tau_1}, \\
Y_t &= Z_t^{0,y,\zeta_0} \ind_{t < \tau_1} + Z_t^{\tau_1, \ell + \eta_1, \zeta_1} \ind_{\tau_1 \leq t < \tau_2} + Z_t^{\tau_2, \ell + \eta_2, \theta_2} \ind_{t \geq \tau_2}.
\end{align*}
We also introduce the random counting measures
\begin{align*}
m(\dd t \, \dd e) &\coloneqq \delta_{\{\tau_1, \zeta_1\}}(\dd t \, \dd e),
&
p^m_t &\coloneqq \ind_{\tau_1 \leq t},
\\
n(\dd t \, \dd z) &\coloneqq \delta_{\{\tau_2, \theta_2\}}(\dd t \, \dd z),
&
p^n_t &\coloneqq \ind_{\tau_2 \leq t}.
\end{align*}
We get that $p^m$ and $p^n$ are $\bbY$-predictable and the random counting measure $m$ associated to the signal process $X$ admits the $\bbF$-compensator
\begin{equation*}
\mu((0,t] \times \{j\}) = \int_0^t r_{X_{s^-} \, j} \, \dd p_s^m, \quad t > 0, \, j \in E,
\end{equation*}
whence $R^m(t,x; \{j\}) = r_{x(t^-) \, j}$, for all $t > 0$, $x \in \cD_E$, and $j \in E$.
The random counting measure $n$ admits the $\bbF$-compensator
\begin{equation*}
\nu((0,t] \times B) = \int_0^t \int_B R^n(s, X_{s^- \land \cdot}, Y_{s^-}; \, \dd z) \, \dd p^n_s, \quad B \in \cB(\R),
\end{equation*}
where $R^n(t,x,y; B) \coloneqq \widetilde R^n(t, x(t^-), y; B)$, for all $t > 0$, $x \in \cD_E$, $y \in \R$, and $B \in \cB(\R)$.
This implies that Assumption~\ref{hp:jointlaw}-\eqref{hp:predproj} is verified.

For all $t > 0$, $x \in \cD_E$, $y \in \R$, we consider the function $b(t,x,y) \coloneqq \widetilde b(t,x(t),y)$, and arguing as in Example~\ref{ex:ODEjumps}, we find that the observed process $Y$ satisfies the SDE
\begin{equation*}
\left\{
\begin{aligned}
\dd Y_t &= b(t, X_{t \land \cdot}, Y_t) \, \dd t + \int_E K^p(t,X_{t^- \land \cdot},e) \, m(\dd t \, \dd e) + \int_{\R} G^p(t,X_{t^- \land \cdot},Y_t,z) \, n(\dd t \, \dd z), \quad t \geq 0, \\
Y_0 &= y \in \R,
\end{aligned}
\right.
\end{equation*}
and the hypotheses on the coefficients $\widetilde b$, $K^p$, and $G^p$ ensure that Assumptions~\ref{hp:jointlaw}-\eqref{hp:Yunique}\eqref{hp:m_n_disjoint}, \ref{hp:sdeexistuniq}, and \ref{hp:martrepr} are verified. Note that, in this case, $m=m^p$ and $n=n^p$, where $m^p$, $n^p$ are the random counting measures appearing in~\eqref{eq:Yode}.

Thanks to Remark~\ref{rem:filternodiff}, we deduce that, for any $T > 0$ and for any bounded and measurable function $\phi \colon \cD_E \rightarrow \R$ the filter $\pi$, defined in \eqref{eq:pidef}, satisfies the SDE
\begin{align*}
&\pi_t(\phi) = \pi_0(\phi)\\
& + \ind_{t \geq \tau_1} \left\{\pi_{\tau_1^-}\bigl(\dA_{\tau_1} \phi\bigr) + U_{\tau_1}^{p,m}(\eta_1) - \int_E \pi_{\tau_1^-}\bigl(\phi(J_{\widetilde \cD_E}(\cdot,\tau_1,e)) \, R^m(\tau_1, \cdot; \dd e)\bigr) + \pi_{\tau_1^-}(\phi)\right\} 
\\
&+ \ind_{t \geq \tau_2} \left\{U_{\tau_2}^{p,n}(\eta_2) - \int_Z \pi_{\tau_2^-}\bigl(\phi(\cdot) \,R^n(\tau_2, \cdot, Y_{\tau_2^-}; \dd z)\bigr) + \pi_{\tau_2^-}(\phi)\right\}, \quad t \in [0,T],
\end{align*}
where
\begin{equation*}
\dA_t \phi(x) \coloneqq \sum_{j \in E} \left[\phi(J_{\widetilde \cD_E}(x,t,j)) - \phi(x)\right] \, r_{x(t^-) \, j}, \quad x \in \widetilde \cD_E, \, t > 0.
\end{equation*}

According to the definitions in Section \ref{sec:filterpi}, after some computations and simplifications, the previous expression can be rewritten as
\begin{equation*}
\pi_t(\phi) =
	\begin{dcases}
		 \pi_0(\phi), & 0 \leq t < \tau_1\\
		\frac{\dd \eta_\phi^i}{\dd \eta^i}(\tau_1,\eta_1)+ \frac{\dd \rho_\phi^i}{\dd \eta^i}(\tau_1,\eta_1), &  \tau_1 \leq t < \tau_2\\
		 \frac{\dd \eta_\phi^i}{\dd \eta^i}(\tau_2,\eta_2), &   t \geq  \tau_2.
	\end{dcases}
\end{equation*}
We immediately deduce from the expression above that in this example the filter is a piecewise constant process.
\end{example}
}

\renewcommand\thesection{\sc Appendix}
\section{}
\renewcommand\thesection{\Alph{subsection}}
\renewcommand\thesubsection{\Alph{subsection}}

\subsection{Proof of some technical results of Section \ref{sec:markov}}

\subsubsection{Proof of Proposition \ref{P_S}.}\label{Section_Appendix_2}
Let $h \in H$ and set $t_0 = 0$. We uniquely define a trajectory $x \in \widetilde \cD_E$ setting for each $t \geq 0$
\begin{align*}
S(h) = S(e_0, t_1, e_1, \dots) \coloneqq
\begin{dcases}
\sum_{k =0}^{n-1} e_k 1_{[t_k, t_{k+1})}(\cdot) + e_n 1_{[t_n, + \infty)}(\cdot), &\text{if } n \coloneqq N(h) < \infty, \\
\sum_{k =0}^{\infty} e_k 1_{[t_k, t_{k+1})}(\cdot), &\text{if } N(h) = \infty,
\end{dcases}
\end{align*}
where $N$ is the map given in \eqref{eq:N}.

Consider, now, $x \in \widetilde \cD_E$. We set $t_0 = 0$, $e_0 \coloneqq x(0)$, and, for each $n \in \N$,
\begin{align*}
t_n &\coloneqq
\begin{dcases}
\inf\{t > t_{n-1}: x(t) \neq x(t_{n-1})\}, &\text{if } \{\cdots\} \neq \emptyset, \\
+\infty,	&\text{otherwise},
\end{dcases}
&
e_n &\coloneqq
\begin{dcases}
x(t_n),	&\text{if } t_n < +\infty, \\
\delta,	&\text{if } t_n = +\infty.
\end{dcases}
\end{align*}
Then, $S^{-1}(x) = (e_0, t_1, e_1, \dots)$ uniquely defines a trajectory in $H$.

The map $S$ is bijective.
Its measurability can be established introducing, first, the following metric $d$ on $H$:
\begin{equation*}
d(h, h') =
\begin{dcases}
\sum_{n = 0}^N \dfrac{1}{2^{n+2}} \left[\rho(|t_n - t_n'|) + \rho(d_E(e_n,e_n'))\right], &\text{if } N(h) = N(h') = N, \\
1,	&\text{otherwise,}
\end{dcases}
\end{equation*}
where $\rho(x) = x (1+x)^{-1}$, $x \geq 0$, $d_E$ is the metric on $E$, and we set $t_0 = t_0' = 0$. Notice that $d(h,h') \leq 1$ for all $h, h' \in H$.
This metric makes $H$ a complete and separable metric space. It is possible to prove that $S$ is continuous on $H$ and, therefore, measurable with respect to the Borel $\sigma$-algebra $\cB(H)$.
Also $S^{-1}$ is measurable by Kuratowski's theorem (\citep[Prop.7.15]{bertsekas:stochoptcontrol}).

Finally, since under Assumptions~\ref{hp:jointlaw}-\eqref{eq:lambdambdd},\eqref{eq:ptmbdd} of  the signal process $X$ is $\bbP$-a.s. non-explosive, we have that $X_{t \land \cdot} \in \widetilde \cD_E, \, \bbP$-a.s. for any $t \geq 0$, hence the equalities between the stopped and the history processes follow.
\qed%

\subsubsection{Proof of Proposition \ref{prop:mYFcomp}}\label{Section_Appendix_5}
{We show that the equality
\begin{equation*}
	\bbE \biggl[\int_0^t \int_{\R} C_s(y) \, m^Y(\dd s \, \dd y)\biggr] = \bbE \biggl[\int_0^t \int_{\R} C_s(y) \, \bfmu^Y(\dd s \, \dd y)\biggr],
\end{equation*}	
holds for any $t \geq 0$,  and any nonnegative and $\bbF$-predictable random field $C \colon \Omega \times  [0,+\infty) \times \R \to \R$. To do this, we consider  $C$ of the form $C_t(y) = \gamma_t \, \ind_B(y)$, where $\gamma$ is an $\bbF$-predictable process and
$B \in \cB(\R)$.
	
We recall that $\bfmu$ defined in~\eqref{eq:bfmu} is the $\bbF$-dual predictable projection of the random counting measure $\bfm$ in~\eqref{m}. Thanks to Proposition \ref{P_S} we can rewrite the $\bbF$-dual predictable projection $\nu$ in \eqref{eq:nu} of the measure $n$ as
\begin{align}
	\nu((0,t] \times B) &= \int_0^t \int_B \bfQ^n(s, \bfX_{s^-}, Y_{s^-}; \dd z)\, \bflambda^n(s, \bfX_{s^-}, Y_{s^-}) \, \dd s\notag\\
	& +\int_0^t \int_B \bfR^n(s, \bfX_{s^-}, Y_{s^-}; \dd z) \, \dd p^n_s, \quad t \geq 0, \, B \in \cB(Z).\label{eq:nubold}
\end{align}
where
$\bfQ^n$, $\bfR^n$ and $\bflambda^n$ are defined in~\eqref{eq:obscoeffs}.

Therefore we get that
\begin{align}
&\bbE \biggl[\int_0^t \int_{\R} C_s(y) \, m^Y(\dd s \, \dd y)\biggr] \\
&= \bbE \biggl[ \int_0^t \gamma_s \int_H \ind_{D^{i,K}_s(B)}(h) \, \bfm^i(\dd s \, \dd h) + \int_0^t \gamma_s \int_Z \ind_{D_s^{i,G}(B)}(z) \, n^i(\dd s \, \dd z),
\\
&\quad + \int_0^t \gamma_s \int_H \ind_{D^{p,K}_s(B)}(h) \, \bfm^p(\dd s \, \dd h)\biggr] + \int_0^t \gamma_s \int_Z \ind_{D^{p,G}_s(B)}(z) \, n^p(\dd s \, \dd z) \\
&=\bbE \biggl[\int_0^t \gamma_s \int_H \ind_{D_s^{i,K}(B)}(h) \, \bflambda^m(s, \bfX_{s^-}) \, \bfQ^m(s, \bfX_{s^-} ; \dd h)\bigr) \dd s \\
&\quad+\int_0^t \gamma_s  \int_Z \ind_{D_s^{i,G}(B)}(z) \, \bflambda^n(s, \bfX_{s^-}, Y_{s-}) \, \bfQ^n(s, \bfX_{s^-} , Y_{s-}; \dd z)\bigr) \dd s \\
&\quad+ \int_0^t \gamma_s\int_H \ind_{D_s^{p,K}(B)}(h) \, \,\bfR^m(s, \bfX_{s^-}; \dd h)\bigr) \,\dd p_s^m \\
&\quad+ \int_0^t \gamma_s\int_Z \ind_{D_s^{p,G}(B)}(z) \, \,\bfR^n(s, \bfX_{s^-}, Y_{s-}; \dd z)\bigr) \,\dd p_s^n \biggr],
\end{align}
which concludes the proof.
\qed}

\subsection{Proof of Theorem \ref{T:mtg_rep}}\label{Section_Appendix_7}

We partly follow the proof of \citep[Proposition 2.4]{cecicolaneri:ks}.
{\color{blue}%
We show, first, that $(\cY_t)_{t \in [0,T]} = \bbG^T \vee \bbH^T$. On the one hand, we have that $\widetilde W$ is a $((\cY_t)_{t \in [0,T]}, \widetilde \bbP)$-Brownian motion and that $m^Y$ is the jump measure of $Y$. Therefore, $\bbG^T \vee \bbH^T \subseteq (\cY_t)_{t \in [0,T]}$.%
}
On the other hand, since $Y$ satisfies an SDE driven by $m^Y$ and  $\widetilde W$, namely
\begin{equation*}
	\left\{
	\begin{aligned}
		&\dd Y_t = \sigma(t, Y_t) \, \dd \widetilde W_t + \int_\R z \, m^Y(\dd t \, \dd z), & &t \in [0,T], \\
		&Y_0 = y \in \R,
	\end{aligned}
	\right.
\end{equation*}
the converse implication holds, and hence we get that $(\cY_t)_{t \in [0,T]} = \bbG^T \vee \bbH^T$.
{\color{blue}%
Therefore, Assumption~\ref{hp:martrepr} ensures that $(\cY_t)_{t \in [0,T]}$ is right-continuous and }that any $((\cY_t)_{t \in [0,T]}, \widetilde \bbP)$-local martingale $\widetilde M = (\widetilde M_t)_{t \in [0,T]}$ admits the representation
\begin{align*}
\widetilde M_t= \widetilde M_0 + \int_0^t \widetilde \gamma_s \, \dd \widetilde W_s + \int_0^t \int_\R \widetilde U_s(y) (m^Y - \hat \bfmu^Y) (\dd s \,\dd y), \quad t \in [0,T],
\end{align*}
where $\widetilde \gamma$ and $\widetilde U$ are respectively a $(\cY_t)_{t \in [0,T]}$-predictable process and a $(\cY_t)_{t \in [0,T]}$-predictable random field, satisfying
\begin{equation}
\int_0^T |\widetilde \gamma_s|^2 \dd s < \infty, \quad
\int_0^T \int_\R |\widetilde U_s(y)| \, \hat \bfmu^Y(\dd s \, \dd y) < \infty, \quad \bbP\text{-a.s.}
\end{equation}
Moreover, by~\citep[Th. 11.17]{he:semimart} (see also~\citep[Th. III.4.20]{jacod2013:limit}) we can choose $\widetilde U$ so that the process
\begin{equation*}
\hat{\widetilde U_t} \coloneqq \int_\R \widetilde U_t(y) \, \ind_{t \leq T} \, \hat \bfmu^Y(\{t\} \times \dd y), \quad t \geq 0,
\end{equation*}
satisfies $\{(\omega, t) \in \Omega \times (0,+\infty) \colon \ind_{t \leq T} \, \hat \bfmu^Y(\omega; \, \{t\} \times \R) = 1\} \subset \{(\omega, t) \in \Omega \times (0,+\infty) \colon \hat {\widetilde U_t} = 0\}$.
Notice that, given the structure of $\hat \bfmu^Y$ in~\eqref{eq:mYYcomp}, we also have that:
\begin{multline*}
\{(\omega, t) \in \Omega \times (0,+\infty) \colon \ind_{t \leq T} \, \hat \bfmu^Y(\omega; \, \{t\} \times \R) = 1\} \\
= \{(\omega, t) \in \Omega \times (0,+\infty) \colon \ind_{t \leq T} \, \hat \bfmu^Y(\omega; \, \{t\} \times \R) > 0\} = J_Y \cap (\Omega \times [0,T]),
\end{multline*}
where $J_Y \coloneqq \{(\omega,t) \in \Omega \times [0,+\infty) \colon \hat \bfmu^Y(\{t\} \times \R) > 0\}$.

Next, let $M = (M_t)_{t \in [0,T]}$ be a $((\cY_t)_{t \in [0,T]}, \bbP)$-local martingale.
Then $M_t=\widetilde M_t \, L_t, \, t \in [0,T]$, for some $((\cY_t)_{t \in [0,T]}, \widetilde \bbP)$-local martingale $\widetilde M = (\widetilde M_t)_{t \in [0,T]}$.
Therefore, thanks to the previous representation result, applying Ito's product rule we have that:
\begin{align*}
M_t
&= \widetilde M_t \, L_t = \int_0^t \widetilde M_{s^-} \, \dd L_s + \int_0^t L_{s^-} \, \dd \widetilde M_s + [\widetilde M, Z]_t \\
&= \int_0^t \Bigl[L_s \, \widetilde \gamma_s - \widetilde M_s \frac{\Pi_s\bigl(\bfb(s, \cdot, Y_s)\bigr)}{\sigma_s}\Bigr] \, \dd I_s + \int_0^t \int_\R L_{s^-} \widetilde U_s(y) [m^Y - \bfmu^Y](\dd s \, \dd y).
\end{align*}

To get the representation result in~\eqref{eq:martrepr}, it is enough defining $\gamma_t \coloneqq L_t \, \widetilde \gamma_t - \widetilde M_t \frac{\Pi_t\bigl(\bfb(t, \cdot, Y_t)\bigr)}{\sigma_t}$, $t \in [0,T]$, and $U_t(\cdot) \coloneqq L_{t^-} \widetilde U_t(\cdot)$, $t \geq 0$. We observe that, since all the processes defining $\gamma$ are $(\cY_t)_{t \in [0,T]}$ adapted and right-continuous, we have that $\gamma$ is $(\cY_t)_{t \in [0,T]}$-progressively measurable; Moreover, it is immediate to deduce that $U$ is a $(\cY_t)_{t \in [0,T]}$-predictable random field. Both $\gamma$ and $U$ satisfy the integrability conditions in~\eqref{int_cond}, by standard computations. Since the process $\hat U_t$ defined in~\eqref{eq:hatU} satisfies $\hat U_t = L_{t^-} \hat{\widetilde U_t}, \, t \geq 0$ and $L$ is strictly positive, we have also that
\begin{equation*}
J_Y \cap (\Omega \times [0,T]) \subset \{(\omega, t) \in \Omega \times (0,+\infty) \colon \hat{\widetilde U_t} = 0\} = \{(\omega, t) \in \Omega \times (0,+\infty) \colon \hat U_t = 0\}. \qed
\end{equation*}


\subsection{Proof of Lemma \ref{L: martY}}\label{Section_Appendix_6}
Let us define for any $t \geq 0$
\begin{align*}
\Psi_t&:=\bbE \biggl[\int_0^t \int_H [f(h)- f(\bfX_{s-})]\,\bfmu(\dd s\, \dd h) \biggm |\cY_t\biggr],
&
\Theta_t&:= \int_0^t \Pi_{s^-}\bigl(\cL_s f\bigr) \, \dd s + \int_0^t \Pi_{s^-}\bigl(\cA_s f\bigr) \, \dd p^m_s.
\end{align*}
Thanks to boundedness of $f$ and Assumption~\ref{hp:jointlaw}, $\Psi$ is a well defined process and $\bbE|\Psi_t| < +\infty$, $\bbE|\Theta_t| < +\infty$ for all $t \geq 0$. Moreover $M_t^f = \Psi_t + \Theta_t$, $t \geq 0$.

Therefore, to show that $M^f$ is a $\bbY$-martingale it is sufficient to prove that for any $0 \leq u \leq t$ we have
\begin{align*}
\bbE[\Psi_t|\cY_u] &= \Psi_u, & \bbE[\Theta_t|\cY_u] &= \Theta_u.
\end{align*}

Fix $0 \leq u \leq t < +\infty$. By conditioning to $\cY_u$, we get that
\begin{align*}
\bbE[\Psi_t|\cY_u]&=\Psi_u+\bbE \biggl[\int_u^t \int_H [f(h)- f(\bfX_{s-})]\,\bfmu(\dd s\, \dd h) \biggm |\cY_u\biggr], \\
\bbE[\Theta_t|\cY_u]&=\Theta_u+\bbE\biggl[ \int_u^t \Pi_{s^-}\bigl(\cL_s f\bigr) \, \dd s\bigg |\cY_u\biggr] + \bbE\biggl[ \int_u^t \Pi_{s^-}\bigl(\cA_s f\bigr) \, \dd p^m_s\biggm |\cY_u\biggr].
\end{align*}
Now we have that
\begin{align*}
&\mathrel{\phantom{=}} \bbE \biggl[\int_u^t \int_H [f(h)- f(\bfX_{s-})]\,\bfmu(\dd s\, \dd h) \biggm |\cY_u\biggr] \\
&= \bbE \biggl[\int_u^t \int_H \cL_s f(\bfX_{s^-}) \, \dd s \biggm |\cY_u\biggr] + \bbE \biggl[\int_u^t \int_H \cA_s f(\bfX_{s^-}) \, \dd p^m_s \biggm |\cY_u\biggr]
\\
&=\bbE\biggl[ \int_u^t \Pi_{s^-}\bigl(\cL_s f\bigr) \, \dd s\biggm |\cY_u\biggr] + \bbE\biggl[ \int_u^t \Pi_{s^-}\bigl(\cA_s f\bigr) \, \dd p^m_s\biggm |\cY_u\biggr],
\end{align*}
where the last equality is justified by the fact that $\Pi_{t^-}\bigl(\cL_t f\bigr), \, t \geq 0$, and $\Pi_{t^-}\bigl(\cA_t f\bigr) \, \dd p^m_t, \, t \geq 0$, are the $\bbY$-optional projections of $\cL_t f(\bfX_{t^-}), \, t \geq 0$, and $\cA_t f(\bfX_{t^-}), \, t \geq 0$, respectively, and an application of~\citep[Th. 5.16]{he:semimart} (see also \citep[Ch. VI, (58.3)]{dellacheriemeyer:B}).
\qed

\bibliographystyle{plainnat}
\bibliography{Bibliography}
\end{document}